\documentclass[12pt]{amsart}
\usepackage{amssymb,latexsym,graphicx,amscd}
\input xy
\xyoption{all}

\numberwithin{equation}{section}

\newtheorem{theorem}{Theorem}[section]
\newtheorem{proposition}[theorem]{Proposition}
\newtheorem{conjecture}[theorem]{Conjecture}
\newtheorem{corollary}[theorem]{Corollary}
\newtheorem{lemma}[theorem]{Lemma}

\theoremstyle{definition}

\newtheorem{remark}[theorem]{Remark}
\newtheorem{example}[theorem]{Example}

\newcommand{\Affine}{{\mathbb A}}

\newcommand{\End}{\operatorname{End}}

\newcommand{\Gr}{\operatorname{Gr}}

\newcommand{\im}{\operatorname{im}}
\newcommand{\Qsf}{\Q_{\,\rm sf}}

\newcommand{\Ext}{\operatorname{Ext}}

\newcommand{\Hom}{\operatorname{Hom}}

\newcommand{\C}{{\mathbb C}}
\newcommand{\Q}{{\mathbb Q}}
\newcommand{\Z}{{\mathbb Z}}
\newcommand{\GL}{\operatorname{GL}}

\newcommand{\rank}{\operatorname{rank}}
\newcommand{\PP}{{\bf P}}
\newcommand{\PPP}{{\mathbb P}}
\newcommand{\TT}{\mathbb{T}}
\newcommand{\bb}{\mathbf{b}}
\newcommand{\dd}{\mathbf{d}}
\newcommand{\e}{\mathbf{e}}
\newcommand{\g}{\mathbf{g}}
\newcommand{\h}{\mathbf{h}}
\newcommand{\yy}{\mathbf{y}}
\newcommand{\M}{{\mathcal M}}
\newcommand{\N}{{\mathcal N}}
\newcommand{\Pcal}{{\mathcal P}}
\newcommand{\proj}{{\rm proj}}
\newcommand{\inj}{{\rm inj}}

\newcommand{\Trop}{\operatorname{Trop}}
\newcommand{\Rep}{\operatorname{Rep}}
\renewcommand{\eqref}[1]{{\rm (\ref{#1})}}
\newcommand{\E}{{\mathcal E}}
\newcommand{\overunder}[2]{{\textstyle\frac{\ \ #1\ \ }{\ \ #2\ \ }}}
\newcommand{\twobyone}[2]{
\begin{bmatrix}#1\\#2
\end{bmatrix}
}

\begin{document}

\title[Quivers with potentials~II]
{Quivers with potentials and their representations~II: applications to cluster algebras}

\author{Harm Derksen, Jerzy Weyman and Andrei Zelevinsky}




\begin{abstract}
We continue the study of quivers with potentials 
and their representations initiated in the first paper of the series.
Here we develop some applications of this theory to cluster algebras.
As shown in the ``Cluster algebras~IV" paper, the cluster algebra structure is to a
large extent controlled by a family of integer vectors called
\emph{$\g$-vectors}, and a family of integer polynomials called \emph{$F$-polynomials}.
In the case of skew-symmetric exchange matrices
we find an interpretation of these $\g$-vectors and $F$-polynomials in terms of
(decorated) representations of quivers with potentials.
Using this interpretation, we prove most of the
conjectures about $\g$-vectors and $F$-polynomials made in loc. cit.
\end{abstract}

\date{April 3, 2009; revised November 13, 2009; final version March 23, 2010}

 \thanks{Research of H.~D. supported
 by the NSF grants DMS-0349019 and DMS-0901298.
Research of J.~W. supported
 by the NSF grant DMS-0600229.
Research of  A.~Z. supported by the
 NSF grants DMS-0500534 and DMS-0801187, and by a Humboldt Research Award.}

\maketitle

\tableofcontents{}

\section{Introduction}
\label{sec:intro}

This paper continues our study of quivers with potentials 
and their representations initiated in \cite{dwz}.
Here we develop some applications of this theory to the theory of cluster algebras.
As shown in \cite{ca4}, the structure of cluster algebras is to a
large extent controlled by a family of integer vectors called
\emph{$\g$-vectors}, and a family of integer polynomials called \emph{$F$-polynomials}.
In the case of skew-symmetric exchange matrices (the terminology will be recalled later),
we find an interpretation of $\g$-vectors and $F$-polynomials in terms of
representations of quivers with potentials.
Using this interpretation, we prove most of the
conjectures about $\g$-vectors and $F$-polynomials made in \cite{ca4}.

Now we describe the main results of the paper in more detail.
Fix a positive integer~$n$.
As in \cite{ca1} and \cite[Definition~2.8]{ca4}, we work with the \emph{$n$-regular tree}~$\TT_n$
whose edges are labeled by the numbers $1, \dots, n$,
so that the $n$ edges emanating from each vertex receive different labels.
We write $t \overunder{k}{} t'$
to indicate that vertices $t, t' \in\TT_n$ are joined by an edge labeled by~$k$.
We also fix a vertex $t_0 \in \TT_n$ and a skew-symmetrizable
integer $n \times n$ matrix $B = (b_{i,j})$ (recall that this
means that $d_i b_{i,j} = - d_j b_{j,i}$ for some positive integers $d_1, \dots, d_n$).
We refer to~$B$ as the \emph{exchange matrix} at $t_0$.
To $t_0$ and $B$ we associate a family of integer vectors
$\g_{\ell;t} = \g_{\ell;t}^{B;t_0} \in \Z^n$
(\emph{$\g$-vectors}) and a family of integer polynomials $F_{\ell;t} =
F_{\ell;t}^{B;t_0} \in \Z[u_1,\dots,u_n]$ (\emph{$F$-polynomials}) in
$n$ independent variables $u_1, \dots, u_n$; here $\ell = 1, \dots, n$, and $t \in \TT_n$.
Both families can be defined via the recurrence relations on the tree $\TT_n$
given by \eqref{eq:gg-initial} -- \eqref{eq:deg-mut2} and
\eqref{eq:F-initial} -- \eqref{eq:F-mut2} below.

Now we state some conjectures from \cite{ca4}.

\begin{conjecture}
(\cite[Conjecture~5.4]{ca4})
\label{con:F-CT-1}
Each polynomial $F_{\ell;t}^{B;t_0}$ has constant term~$1$.
\end{conjecture}

In view of \cite[Proposition~5.3]{ca4}, Conjecture~\ref{con:F-CT-1}
is equivalent to the following.

\begin{conjecture}(\cite[Conjecture~5.5]{ca4})
\label{con:F-HT-1}
Each polynomial $F_{\ell;t}^{B;t_0}$ has a unique
monomial of maximal degree.
Furthermore, this monomial has coefficient~$1$, and it is
divisible by all the other occurring monomials.
\end{conjecture}

\begin{conjecture} (\cite[Conjecture~6.13]{ca4})
\label{con:signs-gi}
For every~$t \in \TT_n$,
the vectors $\g_{1;t}^{B;t_0}, \dots, \g_{n;t}^{B;t_0}$ are
sign-coherent, i.e., for any $i = 1, \dots, n$, the $i$-th components of all
these vectors are either all nonnegative, or all nonpositive.
\end{conjecture}

\begin{conjecture} (\cite[Conjecture~7.10(2)]{ca4})
\label{con:g-vector-Z-basis}
For every~$t \in \TT_n$, the vectors
$\g_{1;t}^{B;t_0}, \dots, \linebreak \g_{n;t}^{B;t_0}$
form a $\Z$-basis of the lattice~$\Z^n$.
\end{conjecture}

\begin{conjecture} (\cite[Conjecture~7.10(1)]{ca4})
\label{con:g-vectors-separate-cluster-monomials}
Suppose we have
$$\sum_{i \in I} a_i \g_{i;t}^{B;t_0} = \sum_{i \in I'} a'_i \g_{i;t'}^{B;t_0}$$
for some $t, t' \in \TT_n$, some nonempty subsets $I, I' \subseteq \{1, \dots, n\}$ and some
positive integers $a_i$ and $a'_i$.
Then there is a bijection $\sigma: I \to I'$ such that, for every
$i \in I$, we have
$$a_i = a'_{\sigma(i)}, \quad \g_{i;t}^{B;t_0} =
\g_{\sigma(i);t'}^{B;t_0}, \quad F_{i;t}^{B;t_0} =
F_{\sigma(i);t'}^{B;t_0}.$$
In particular, for given $B$ and $t_0$, each polynomial
$F_{i;t}^{B;t_0}$ is determined by the vector $\g_{i;t}^{B;t_0}$.
\end{conjecture}

To state our last conjecture, we need to recall the \emph{matrix
mutation} introduced in \cite{ca1}.
For
any $k = 1, \dots, n$, we define an integer $n \times n$ matrix $\mu_k(B) = (b'_{i,j})$
by setting
\begin{equation}
\label{eq:B-mutation}
b'_{i,j} =
\begin{cases}
-b_{i,j} &  \text{if $i=k$ or $j=k$;} \\[.05in]
b_{i,j} + [b_{i,k}]_+ \ [b_{k,j}]_+ \ - [-b_{i,k}]_+\ [-b_{k,j}]_+
 & \text{otherwise,}
\end{cases}
\end{equation}
where we use the notation
\begin{equation}
\label{eq:x+}
[b]_+ = \max(b,0).
\end{equation}

\begin{conjecture} (\cite[Conjecture~7.12]{ca4})
\label{con:g-transition}
Let $t_0 \overunder{k}{} t_1$ be two adjacent vertices in~$\TT_n$,
and let $B' = \mu_k(B)$.
Then, for any $t \in \TT_n$ and $\ell = 1, \dots, n$,
the $\g$-vectors $\g_{\ell;t}^{B;t_0} = (g_1, \dots, g_n)$
and $\g_{\ell;t}^{B';t_1} = (g'_1, \dots, g'_n)$ are related as
follows:
\begin{equation}
\label{eq:Langlands-dual-trop}
g'_j =
\begin{cases}
-g_k  & \text{if $j = k$};\\[.05in]
g_j + [b_{j,k}]_+ g_k
  - b_{j,k} \min(g_k,0)
 & \text{if $j \neq k$}.
\end{cases}
\end{equation}
\end{conjecture}

We can now state one of our main results.

\begin{theorem}
\label{th:fg-conjectures}
The conjectures \ref{con:F-CT-1} -- \ref{con:g-transition} hold
under the assumption that the exchange matrix~$B$ is
skew-symmetric.
\end{theorem}

\begin{remark}
As explained in \cite[Remark~7.11]{ca4}, Conjectures~\ref{con:F-CT-1}
and \ref{con:g-vectors-separate-cluster-monomials} imply the
linear independence of cluster monomials in any cluster algebra
satisfying a mild additional condition \cite[(7.10)]{ca4}.
\end{remark}

\begin{remark}
The above conjectures were established in \cite{fk} under some
additional conditions (that the cluster algebras in question admit
a certain categorification).
Our method described below has an advantage that the only
condition we need is that the matrix~$B$ is skew-symmetric.
\end{remark}

As mentioned already, our proof of Theorem~\ref{th:fg-conjectures}
is based on interpreting $\g$-vectors and $F$-polynomials in terms
of representations of quivers with potentials.
First of all, a skew-symmetric integer $n \times n$ matrix~$B$ can
be encoded by a quiver $Q(B)$ without loops and oriented $2$-cycles on
the set of vertices $[1,n]=\{1, \dots, n\}$.
This is done as follows:
\begin{equation}
\label{eq:QB}
\text{for any two vertices $i \neq j$ there
are $[b_{i,j}]_+$ arrows from $j$ to $i$ in $Q(B)$.}
\end{equation}
As is customary these days, we represent a quiver by a quadruple
$(Q_0,Q_1,h,t)$ consisting of a pair of finite
sets $Q_0$ (\emph{vertices}) and $Q_1$ (\emph{arrows}) supplied
with two maps $h:Q_1 \to Q_0$ (\emph{head}) and $t:Q_1 \to Q_0$
(\emph{tail}); every arrow $a \in Q_1$ is viewed as a directed edge
$a: t(a) \rightarrow h(a)$.
For the quiver $Q(B)$, the vertex set $Q_0$ is identified with $[1,n]$.

Recall that a representation~$M$ of a quiver~$Q$ is specified by a family of
finite-dimensional vector spaces $(M(i))_{i \in Q_0}$ (for
simplicity we work over~$\C$) and a family of linear maps
$a = a_M: M(t(a)) \to M(h(a))$ for $a \in Q_1$.
The \emph{dimension vector}~$\dd_M$ of~$M$ is given by
\begin{equation}
\label{eq:dim-vector}
\dd_M=(\dim M(1), \dots, \dim M(n)).
\end{equation}
For every integer vector $\e = (e_1, \dots, e_n)$, we denote by
$\Gr_\e(M)$ the \emph{quiver Grassmannian} of
subrepresentations  $N \subseteq M$ with $\dd_N = \e$.
In simple terms, an element of $\Gr_\e(M)$ is an $n$-tuple
$(N(1), \dots, N(n))$, where each $N(i)$ is a subspace of dimension
$e_i$ in $M(i)$, and $a_M (N(j)) \subseteq N(i)$ for any arrow $a: j \to i$.
Thus, $\Gr_\e(M)$ is a closed subvariety of the product of ordinary
Grassmannians $\prod_{i = 1}^n \Gr_{e_i}(M(i))$, hence a
projective algebraic variety.

Let $\chi(\Gr_\e (M))$ denote the Euler-Poincar\'e characteristic
of $\Gr_\e (M)$ (see e.g., \cite[Section~4.5]{fulton}).
We associate to a quiver representation~$M$ the polynomial
$F_M \in \Z[u_1, \dots, u_n]$ given by
\begin{equation}
\label{eq:F-M}
F_M(u_1, \dots, u_n) =
\sum_\e \chi(\Gr_\e (M)) \prod_{i=1}^n u_i^{e_i}.
\end{equation}
We refer to $F_M$ as the \emph{$F$-polynomial} of~$M$.

It is immediate from \eqref{eq:F-M} that every polynomial $F_M$
satisfies properties in Conjectures~\ref{con:F-CT-1} and \ref{con:F-HT-1}.
Thus, to prove these conjectures for any skew-symmetric matrix~$B$,
it suffices to construct, for $t_0$, $\ell$ and $t$ as above,
a representation~$M = M_{\ell;t}^{B;t_0}$ of $Q(B)$ such that
\begin{equation}
\label{eq:F-quiver-cluster}
F_{\ell;t}^{B;t_0} = F_M.
\end{equation}
We do this in Theorem~\ref{th:g-F-rep} using mutations of quivers with
potentials and their representations introduced and studied in
\cite{dwz}.

To prove the conjectures involving $\g$-vectors, we need to
consider quiver representations equipped with some extra structure.
First, following \cite{mrz}, we work with \emph{decorated representations}
$\M = (M,V)$, where~$M$ is a representation of $Q(B)$, and
$V = (V(i))_{i \in Q_0}$ is a family of
finite-dimensional $\C$-vector spaces, with no maps attached.
Second, $M$ must be \emph{nilpotent}, that is, annihilated by all
sufficiently long paths in~$Q(B)$.
Finally and most importantly, the action of arrows in~$M$ must
satisfy the relations from the Jacobian ideal of a generic potential on~$Q(B)$.
The corresponding setup developed in \cite{dwz} will be recalled in
Section~\ref{sec:QP-background}, here we just describe a general form of the relations.
For every two arrows $a, b \in Q_1$ with $h(a) = t(b)$,
a generic potential~$S$ on~$Q(B)$ gives rise to
an element $\partial_{ba} (S)$ of the complete path
algebra of $Q(B)$: this is a (possibly infinite) linear
combination of paths from $h(b)$ to $t(a)$.
For every $k \in Q_0$,
these elements give rise to the triangle of linear maps
\begin{equation}
\label{eq:triangle}
\xymatrix{
& M(k)\ar@<.5ex>[rd]^{\beta_k} &\\
M_{\rm in}(k)
\ar@<.5ex>[ru]^{\alpha_k} & & M_{\rm out}(k)\ar@<.5ex>[ll]^{\gamma_k}
}
\end{equation}
Here the spaces $M_{\rm in}(k)$
and $M_{\rm out}(k)$ are given by
\begin{equation}
\label{eq:in-out-spaces}
M_{\rm in}(k) = \bigoplus_{h(a) = k} M(t(a)), \quad
M_{\rm out}(k) = \bigoplus_{t(b) = k} M(h(b)),
\end{equation}
the maps $\alpha_k$ and $\beta_k$ are given by
\begin{equation}
\label{eq:in-out-maps}
\alpha_k = \sum_{h(a) = k} a_M, \quad
\beta_k = \sum_{t(b) = k} b_M,
\end{equation}
and, for each $a, b \in Q_1$ with $h(a) = t(b) = k$, the
component $\gamma_{a,b}: M(h(b)) \to M(t(a))$ of $\gamma_k$
is given by
\begin{equation}
\label{eq:gamma-ab}
\gamma_{a,b} = (\partial_{ba} S)_M.
\end{equation}

In these terms, the relations on~$M$ imposed by the choice of $S$
are just the following:
\begin{equation}
\label{eq:triangle-relations}
\alpha_k \circ \gamma_k = 0, \quad \gamma_k \circ \beta_k = 0.
\end{equation}
We refer to a decorated representation with these properties as a
\emph{QP-representation} (for ``quivers with potentials").

Now we define the $\g$-vector $\g_{\M} = (g_1, \dots, g_n) \in \Z^n$ of a
QP-representation~$\M = (M,V)$ by setting
\begin{equation}
\label{eq:g-M}
g_k = \dim \ker \gamma_k - \dim M(k) + \dim V(k) \ .
\end{equation}

As a first step towards proving Conjectures~\ref{con:signs-gi} -
\ref{con:g-transition} for $B$ skew-symmetric, in
Theorem~\ref{th:g-F-rep} we construct, for $t_0$, $\ell$ and $t$ as above,
an indecomposable QP-representation $\M = \M_{\ell;t}^{B;t_0}$ of $Q(B)$ such that
\begin{equation}
\label{eq:g-quiver-cluster}
\g_{\ell;t}^{B;t_0} = \g_{\M}
\end{equation}
(note that $\M = (M,V)$, where the quiver representation~$M = M_{\ell;t}^{B;t_0}$
satisfies \eqref{eq:F-quiver-cluster}).

Our main tool in working with QP-representations is the mutation
operation $\M \mapsto \mu_k (\M)$ (for each $k \in Q_0$) sending
QP-representations of the quiver $Q(B)$ to those of $Q(\mu_k(B))$.
This operation was introduced and studied in \cite{dwz}, where it
was shown in particular that $\mu_k$ sends indecomposable QP-representations
into indecomposable ones.
In terms of the mutations, the family of QP-representations $\M_{\ell;t}^{B;t_0}$
is determined by the following two properties:
\begin{itemize}
\item For $t = t_0$, we have
\begin{equation}
\label{eq:neg-simple}
\M_{\ell;t_0}^{B;t_0} = \mathcal S_\ell^-,
\end{equation}
the \emph{negative simple QP-representation} such that the only
nonzero space among the $M(i)$ and $V(i)$ is $V(\ell) = \C$.

\smallskip

\item If $t_0 \overunder{k}{} t_1$ in~$\TT_n$, and $B' = \mu_k(B)$
then
\begin{equation}
\label{eq:Mlt-mutation}
\M_{\ell;t}^{B';t_1} = \mu_k(\M_{\ell;t}^{B;t_0}).
\end{equation}
\end{itemize}

In contrast with the situation for $F$-polynomials, where the
interpretation \eqref{eq:F-quiver-cluster} immediately implies
Conjectures~\ref{con:F-CT-1} and \ref{con:F-HT-1}, deducing
Conjectures~\ref{con:signs-gi} -- \ref{con:g-transition}
from \eqref{eq:g-quiver-cluster} requires further work.
The main new ingredient is the following integer-valued function
on QP-representations: for a QP-representation ${\mathcal M}=(M,V)$ of a quiver~$Q$,
we define the \emph{$E$-invariant} by
\begin{equation}
\label{eq:Einvariant}
E({\mathcal M})=\dim \Hom_Q (M,M)+\sum_{k=1}^n g_k \dim M(k),
\end{equation}
where $g_k$ is given by \eqref{eq:g-M}, and $\Hom_Q$ stands for
the space of homomorphisms of quiver representations.
In Theorem~\ref{th:E-invariant} we prove that $E(\M)$ is invariant
under mutations, i.e., for every $k$ we have $E(\mu_k(\M))=E(\M)$.
Then it follows from \eqref{eq:neg-simple} and \eqref{eq:Mlt-mutation} that
$E(\M_{\ell;t}^{B;t_0}) = 0$ for all $\ell$ and $t$.

Since the numbers $g_k$ may be negative, it is not {\it a priori} clear
that $E(M)$ takes nonnegative values.
We prove this property in Theorem~\ref{th:E-lower-bound-beta}, establishing the following
much sharper lower bound:
\begin{equation}
\label{eq:E-lower-bound-beta}
E(\M) \geq \sum_{k \in Q_0}(\dim \ker \beta_k \cdot \dim (\ker\gamma_k/\im \beta_k)+
\dim M(k) \cdot \dim V(k)).
\end{equation}
As a consequence, for each $\M$ of the form $\M_{\ell;t}^{B;t_0}$,
the right hand side of \eqref{eq:E-lower-bound-beta} is equal
to~$0$, and this information turns out to be exactly what we need
for proving Conjectures~\ref{con:signs-gi} -
\ref{con:g-transition}.

Note that in  view of \eqref{eq:neg-simple} and
\eqref{eq:Mlt-mutation}, the QP-representations $\M_{\ell;t}^{B;t_0}$
can be characterized as those obtained by a sequence of mutations
from a negative simple representation.
We conjecture that this family coincides with the family of indecomposable
QP-representations $\M$ such that $E(\M) = 0$.
As a possible step towards proving this conjecture, in Section~\ref{sec:homological}
we develop a homological interpretation of~$E(\M)$ in the case where the potential is finite
and the Jacobian algebra is finite dimensional.
This interpretation is based on constructing a projective presentation for QP-representations, see
Proposition~\ref{pr:min-presentation}.

The paper is organized as follows.
Sections~\ref{sec:F-poly} -- \ref{sec:QP-background} are devoted
to preliminaries.
The necessary background on cluster algebras is recalled in Section~\ref{sec:F-poly}.
In Section~\ref{sec:F-QP} we collect some general properties of
$F$-polynomials of quiver representations to be used later.
We conclude this section with two examples, showing that a quiver Grassmannian may be
singular, and that it may have  negative Euler characteristic.
The necessary background from \cite{dwz} on quivers with
potentials (QP's) and their representations is collected in
Section~\ref{sec:QP-background}.

Section~\ref{sec:qp-interpretation} contains the first important new result of the paper
-- Theorem~\ref{th:g-F-rep}.
It asserts that the family of QP-representations recursively
defined by conditions \eqref{eq:neg-simple} and \eqref{eq:Mlt-mutation}
provides a representation-theoretic interpretation given by
\eqref{eq:F-quiver-cluster} and \eqref{eq:g-quiver-cluster}
of $F$-polynomials and $\g$-vectors arising in the theory of cluster algebras.
As a consequence, we obtain in Corollary~\ref{cor:CC-formula} a
formula for cluster variables in the coefficient-free cluster
algebra, which generalizes the Caldero-Chapoton formula in \cite[Theorem~3]{ck}.

In Section~\ref{sec:homs-modulo-confined} we prove
Proposition~\ref{pr:mutation-preserves-homs-mod-k},
a technical result preparing the ground for the later proof of
the invariance under mutations of the function $E(\M)$
given by \eqref{eq:Einvariant}.
Roughly speaking, Proposition~\ref{pr:mutation-preserves-homs-mod-k} says that the mutation
at a vertex~$k$ preserves the space of homomorphisms between any
two QP-representations modulo the homomorphisms ``confined" to~$k$.
This result of independent interest was already established in
\cite[Theorem~7.1]{birs} but the present proof seems to be much simpler.
In the rest of Section~\ref{sec:homs-modulo-confined} we show that
the isomorphism in Proposition~\ref{pr:mutation-preserves-homs-mod-k}
can be stated in a functorial way.

The main result in Section~\ref{sec:E} is
Theorem~\ref{th:E-invariant} establishing in particular the
invariance of $E(\M)$ under mutations.
Another useful result there is
Proposition~\ref{pr:E-invariant-star} saying that $E(\M)$ is
invariant under passing to the dual QP-representation of the opposite QP.

In Section~\ref{sec:inequalities} we prove the bound \eqref{eq:E-lower-bound-beta}
(Theorem~\ref{th:E-lower-bound-beta}).
The proof of Theorem~\ref{th:fg-conjectures} is obtained by
combining this result with the results in the preceding sections;
this is done in Section~\ref{sec:proof-cluster-conjectures}.
The concluding Section~\ref{sec:homological} is devoted to the
above-mentioned homological interpretation of the $E$-invariant of
QP-representations.\\[10pt]

\noindent{\bf  Acknowledgement.}  The authors are grateful to Grzegorz Bobi\'nski for providing useful
references, and to an anonymous referee for several helpful
suggestions.
\section{Background on $\g$-vectors and $F$-polynomials}
\label{sec:F-poly}

First of all, we recall that the same rule as in \eqref{eq:B-mutation}
defines the matrix mutation $\mu_k$
for any integer $m \times n$ matrix $\tilde B = (b_{i,j})$ with $m \geq n$,
and any $k = 1, \dots, n$.
This is an involution on the set of integer $m \times n$ matrices.
We call the top $n \times n$ submatrix~$B$ of $\tilde B$
the \emph{principal part} of~$\tilde B$; then $\mu_k(B)$ is the principal part
of $\mu_k( \tilde B)$.
Note also that, if $B$ is skew-symmetrizable, that is,
$d_i b_{i,j} = - d_j b_{j,i}$ for some positive integers $d_1, \dots, d_n$,
then the same choice of $d_1, \dots, d_n$ makes $\mu_k(B)$ skew-symmetrizable as well.
In particular, if $B$ is skew-symmetric then $\mu_k(B)$ is also skew-symmetric.

We say that a family of $m \times n$ integer matrices $(\tilde B(t)_{t \in \TT_n})$ is
a skew-symmetrizable (resp. skew-symmetric) \emph{matrix pattern} of
format $m \times n$ on $\TT_n$ if the principal part $B(t)$ of each $\tilde B(t)$ is
skew-symmetrizable (resp. skew-symmetric), and we have
$\tilde B(t') = \mu_k(\tilde B(t))$ whenever
$t \overunder{k}{} t'$.
Clearly, such a pattern is uniquely determined by each of its
matrices $\tilde B(t_0)$, which can be chosen arbitrarily with
the only condition that its principal part is skew-symmetrizable (resp. skew-symmetric).

Now choose any skew-symmetrizable $n \times n$ integer matrix
$B$ and any vertex $t_0 \in \TT_n$.
We associate to $B$ and $t_0$ the skew-symmetrizable matrix pattern of
format $2n \times n$ such that $\tilde B(t_0) = (b_{i,j})$ has principal part $B$,
and its bottom part is the $n \times n$ identity matrix, that is,
$b_{n+i,j} = \delta_{i,j}$ for $i,j = 1, \dots, n$; we refer to
this pattern as the \emph{principal coefficients pattern} associated to
$B$ and $t_0$.
Let us denote this pattern simply as $(\tilde B(t) = (b_{i,j}(t)))_{t \in \TT_n}$
(with the understanding that $B$ and $t_0$ are fixed).

Now, according to \cite[Proposition~6.6]{ca4}, the vectors $\g_{\ell;t} = \g_{\ell;t}^{B;t_0}$
can be defined by the initial conditions
\begin{equation}
\label{eq:gg-initial}
\g_{\ell;t_0} = \e_\ell \quad (\ell = 1, \dots, n)
\end{equation}
together with the recurrence relations
\begin{align}
\label{eq:deg-mut1}
\g_{\ell;t'}&=\g_{\ell;t} \quad \text{for $\ell\neq
  k$;}\\
\label{eq:deg-mut2}
\g_{k;t'} &= -\g_{k;t} +\sum_{i=1}^n [b_{i,k}(t)]_+ \g_{i;t}
-\sum_{i=1}^n  [b_{n+i,k}(t)]_+ \bb_i
 \,
\end{align}
for every edge $t \overunder{k}{} t'$ in~$\TT_n\,$. Here $\e_1, \dots, \e_n$ are the unit vectors in $\Z^n$,
and $\bb_1, \dots, \bb_n$ are the columns of $B$.

Similarly, by \cite[Proposition~5.1]{ca4}, the
polynomials $F_{\ell;t} = F_{\ell;t}^{B;t_0}(u_1,\dots,u_n)$
can be defined by the initial conditions
\begin{equation}
\label{eq:F-initial}
F_{\ell;t_0} = 1 \quad (\ell = 1,
\dots, n) \ ,
\end{equation}
together with the recurrence relations
\begin{align}
\label{eq:F-mut1}
F_{\ell;t'}&= F_{\ell;t} \quad \text{for $\ell\neq
  k$;}\\
\label{eq:F-mut2}
F_{k;t'} &= \frac{\prod_{i=1}^n u_i^{[b_{n+i,k}(t)]_+} F_{i;t}^{[b_{i,k}(t)]_+}
+ \prod_{i=1}^n u_i^{[-b_{n+i,k}(t)]_+} F_{i;t}^{[-b_{i,k}(t)]_+}}{F_{k;t}}\,,
\end{align}
for every edge $t \overunder{k}{} t'$ in~$\TT_n\,$.

For instance, if $t_1 \overunder{k}{} t_0\,$, then
$\g_{k;t_1}^{B;t_0} = -\e_k + \sum_{i=1}^n [-b_{i,k}]_+ \e_i$,
and $F_{k;t_1}^{B;t_0} = u_k + 1$.

Here is a specific example for the cluster algebra of type~$A_2$
(cf. \cite[Examples~2.10, 3.4, 6.7]{ca4}).

\begin{example}
\label{example:A2-F-polys}
Let $n=2$. The tree $\TT_2$ is an infinite chain.
We denote its vertices by
$\dots,t_{-1}\,,t_0\,,t_1\,,t_2\,,\dots$,
and label its edges as follows:
\begin{equation}
\label{eq:TT2}
\cdots
\overunder{2}{} t_{-1}
\overunder{1}{} t_0
\overunder{2}{} t_1
\overunder{1}{} t_2
\overunder{2}{} t_3
\overunder{1}{} \cdots \,.
\end{equation}
Let
$B=\left[
\begin{smallmatrix}
0 & 1\\
-1&0
\end{smallmatrix}
\right]$.
The $\g$-vectors $\g_{\ell;t} = \g_{\ell;t}^{B;t_0}$ and $F$-polynomials
$F_{\ell;t} = F_{\ell;t}^{B;t_0}$ are shown in
Table~\ref{table:F-poly-A2} (the last column will be explained later).
\begin{table}[ht]
\begin{equation*}
\begin{array}{|c|c|cc|cc|cc|}
\hline
&&&\\[-4mm]
t & \tilde B(t) & \hspace{7mm}
\g_{1;t} & \g_{2;t} &F_{1;t} & F_{2;t} & \h_{1;t} & \h_{2;t}
 \\[1mm]
\hline
&&&\\[-3mm]
t_0 &
\left[\begin{smallmatrix}0&1\\-1&0\\1&0\\0&1\end{smallmatrix}\right] &
\twobyone{1}{0} & \twobyone{0}{1} & 1 & 1 & \twobyone{0}{0} & \twobyone{0}{0}
\\[4.5mm]
\hline
&&&\\[-3mm]
t_1 & \left[\begin{smallmatrix}0&-1\\1&0\\1&0\\0&-1\end{smallmatrix}\right] &
\twobyone{1}{0} & \twobyone{0}{-1} & 1& u_2+1 & \twobyone{0}{0} & \twobyone{0}{-1}
\\[4.5mm]
\hline
&&&\\[-3mm]
t_2 & \left[\begin{smallmatrix}0&1\\-1&0\\-1&0\\0&-1\end{smallmatrix}\right] &
\twobyone{-1}{0} & \twobyone{0}{-1} & u_1 u_2+u_1+1 & u_2+1 & \twobyone{-1}{0} & \twobyone{0}{-1}
\\[4.5mm]
\hline
&&&\\[-3mm]
t_3 & \left[\begin{smallmatrix}0&-1\\1&0\\-1&0\\-1&1\end{smallmatrix}\right] &
\twobyone{-1}{0} & \twobyone{-1}{1} & u_1 u_2+u_1+1 & u_1+1 & \twobyone{-1}{0} & \twobyone{-1}{0}
\\[4.5mm]
\hline
&&&\\[-3mm]
t_4 & \left[\begin{smallmatrix}0&1\\-1&0\\1&-1\\1&0\end{smallmatrix}\right] &
\twobyone{0}{1} & \twobyone{-1}{1} & 1 & u_1+1 & \twobyone{0}{0} & \twobyone{-1}{0}
\\[4.5mm]
\hline
&&&\\[-3mm]
t_5 &
\left[\begin{smallmatrix}0&-1\\1&0\\0&1\\1&0\end{smallmatrix}\right] &
\twobyone{0}{1} & \twobyone{1}{0} & 1 & 1 & \twobyone{0}{0} & \twobyone{0}{0}
 \\[4.5mm]
\hline
\end{array}
\end{equation*}
\smallskip

\caption{$\g$-vectors, $F$-polynomials, and $\h$-vectors in type~$A_2$}
\label{table:F-poly-A2}
\end{table}

Observing that $\tilde B(t_5)$ is obtained from $\tilde B(t_0)$ by
interchanging the two columns, and comparing $\g$-vectors and $F$-polynomials at
$t_0$ and $t_5$, we obtain the following periodicity property:
$$\g_{\ell;t_{m+5}} = \g_{3-\ell;t_{m}}, \quad
F_{\ell;t_{m+5}}(u_1,u_2) = F_{3-\ell;t_{m}}(u_2,u_1) \quad (m \in \Z) \ .$$
\end{example}

\smallskip

Returning to the general situation, we note that the definition
makes it clear that all $F_{\ell;t}(u_1, \dots, u_n)$ are rational functions
with coefficients in~$\Q$.
The following stronger statement was proven in \cite[Propositions~3.6, 5.2]{ca4}.

\begin{proposition}
\label{pr:F-poly}
Each of the rational functions $F_{\ell;t}(u_1, \dots, u_n)$ is a polynomial with integer
coefficients, which is not divisible by any~$u_i$.
\end{proposition}

We now fix $\ell$ and $t$, and discuss the dependency of $\g_{\ell;t}^{B;t_0}$
and $F_{\ell;t}^{B;t_0}$ on the initial vertex $t_0$ and the initial exchange matrix~$B$.
More precisely, choose some $k \in [1,n]$, and suppose that
$t_0 \overunder{k}{} t_1$ and $B_1 = \mu_k(B)$.
We will relate the vectors $\g_{\ell;t}^{B;t_0}$ and
$\g_{\ell;t}^{B_1;t_1}$, and the  polynomials $F_{\ell;t}^{B;t_0}$ and
$F_{\ell;t}^{B_1;t_1}$.
This requires some preparation.

Recall that a \emph{semifield} $(\PP,\cdot,+)$
is an abelian multiplicative group~$(\PP,\cdot)$ endowed with a binary operation of
addition which is commutative, associative, and
distributive with respect to the multiplication in~$\PP$.
With every finite family of indeterminates $u_1, \dots, u_\ell$ one can
associate two semifields: the \emph{universal} semifield $\Qsf(u_1, \dots, u_\ell)$,
and the \emph{tropical} semifield $\Trop(u_1, \dots, u_\ell)$ (cf. \cite[Definitions~2.1, 2.2]{ca4}).
Recall that $\Qsf(u_1, \dots, u_\ell)$ is the set of all rational
functions in $u_1, \dots, u_\ell$ which can be written as
\emph{subtraction-free} rational expressions,
while $\Trop(u_1, \dots, u_\ell)$ is the multiplicative group of Laurent monomials
$u_1^{a_1} \cdots u_\ell^{a_\ell}$ with the addition~$\oplus$ given by
\begin{equation}
\label{eq:tropical-addition}
\prod_j u_j^{a_j} \oplus \prod_j u_j^{b_j} =
\prod_j u_j^{\min (a_j, b_j)} \,.
\end{equation}

Since \eqref{eq:F-mut2} does not involve subtraction, every
$F$-polynomial $F_{\ell;t}^{B;t_0}(u_1, \dots, u_n)$ belongs to
$\Qsf(u_1, \dots, u_n)$ (although it is still not known in general
whether all these polynomials have positive coefficients).
Note that every subtraction-free rational expression $F(u_1, \dots, u_n)$
(in particular, every $F_{\ell;t}^{B;t_0}$)
can be evaluated at any $n$-tuple of elements
$y_1, \dots, y_n$ of an arbitrary semifield~$\PP$.
We denote the result of this evaluation by $F|_\PP (y_i \leftarrow u_i)$.
Using this notation, we denote by
$\h_{\ell;t}^{B;t_0} = (h_1, \dots, h_n)$ the integer
vector given by
\begin{equation}
\label{eq:h}
x_1^{h_1} \cdots x_n^{h_n} = F_{\ell;t}^{B;t_0}|_{\Trop(x_1, \dots, x_n)}
(x_i^{-1} \prod_{j \neq i} x_j^{[-b_{j,i}]_+} \leftarrow u_i).
\end{equation}

\begin{example}
\label{example:A2-h-vectors}
In the situation of Example~\ref{example:A2-F-polys},
the vectors $\h_{\ell;t} = \h_{\ell;t}^{B;t_0}$ are given in
the last column of Table~\ref{table:F-poly-A2}.
In this case, the formula \eqref{eq:h} for the vector $\h_{\ell;t} = (h_1,h_2)$
takes the form
$$x_1^{h_1} x_2^{h_2} = F_{\ell;t}|_{\Trop(x_1,x_2)}
(x_1^{-1}x_2, x_2^{-1}).$$
For example, since $F_{1;t_2} = u_1 u_2+u_1+1$, we obtain
$$F_{1;t_2}|_{\Trop(x_1,x_2)}(x_1^{-1}x_2, x_2^{-1}) =
x_1^{-1} \oplus x_1^{-1}x_2 \oplus 1 = x_1^{-1},$$
hence $\h_{1;t_2} = \twobyone{-1}{0}$.
\end{example}

Next we recall the $Y$-seeds and their mutations (see
\cite[Definitions~2.3, 2.4]{ca4}).
A (labeled)  \emph{$Y$-seed} in a semifield~$\PP$
is a pair $(\yy, B)$, where
\begin{itemize}
\item
$\yy = (y_1, \dots, y_n)$ is an $n$-tuple
of elements of $\PP$, and
\item
$B = (b_{i,j})$ is an $n\!\times\! n$ skew-symmetrizable integer matrix.
\end{itemize}
The \emph{$Y$-seed mutation} at $k \in [1,n]$ transforms
$(\yy, B)$ into a $Y$-seed $\mu_k(\yy, B) = (\yy', B')$,
where $B' = \mu_k(B)$ is given by \eqref{eq:B-mutation},
and the $n$-tuple $\yy' = (y'_1, \dots, y'_n)$
is given by
\begin{equation}
\label{eq:y-mutation}
y'_i =
\begin{cases}
y_k^{-1} & \text{if $i = k$};\\[.05in]
y_i y_k^{[b_{k,i}]_+}
(y_k + 1)^{- b_{k,i}} & \text{if $i \neq k$}.
\end{cases}
\end{equation}

The following result is immediate from
\cite[Proposition~6.8, formulas~(6.26),(6.28)]{ca4}.

\begin{proposition}
\label{pr:g-F-t0-t1}
Suppose $t_0 \overunder{k}{} t_1$ in $\TT_n$, and
the $Y$-seed $(\yy', B_1)$ in $\Qsf(y_1, \dots, y_n)$
is obtained from $(\yy, B)$ by the mutation at~$k$.
Let $h_k$ (resp.~$h'_k$) be the $k$-th component of the
vector $\h_{\ell;t}^{B;t_0}$ (resp. $\h_{\ell;t}^{B_1;t_1}$).
Then the $\g$-vectors
$\g_{\ell;t}^{B;t_0} = (g_1, \dots, g_n)$ and
$\g_{\ell;t}^{B_1;t_1} = (g'_1, \dots, g'_n)$
are related by
\begin{equation}
\label{eq:g-transition}
g'_j =
\begin{cases}
-g_k  & \text{if $j = k$};\\[.05in]
g_j + [b_{j,k}]_+ g_k
  - b_{j,k} h_k
 & \text{if $j \neq k$}.
\end{cases}
\end{equation}
We also have
\begin{equation}
\label{eq:g-thru-h}
g_k = h_k - h'_k,
\end{equation}
and
\begin{equation}
\label{eq:F-t0-t1}
(y_k + 1)^{h_k}F_{\ell;t}^{B;t_0}(y_1, \dots, y_n)
= (y'_k + 1)^{h'_k}
F_{\ell;t}^{B_1;t_1}(y'_1, \dots, y'_n)\,.
\end{equation}
\end{proposition}

We conclude this section by recalling \cite[Corollary~6.3]{ca4}
that explains why the $\g$-vectors and $F$-polynomials play a
crucial role in the theory of cluster algebras.
Recall that a cluster algebra~$\mathcal A$ is specified by a
choice of a $Y$-seed $(\yy, B)$ in a semifield~$\PP$.
Let $\mathcal F = \Q \PP(x_1, \dots, x_n)$ be the field of rational functions
in commuting independent variables $x_1, \dots, x_n$ over the
quotient field $\Q\PP$ of the integer group ring $\Z \PP$ of the
multiplicative group $\PP$.
Then each $\ell$ and $t$ as above gives rise to a \emph{cluster
variable} $x_{\ell;t} \in \mathcal F$ given by
\begin{equation}
\label{eq:xjt=F/F}
x_{\ell;t} = \frac{F_{\ell;t}^{B;t_0}|_{\mathcal F}(\hat y_1, \dots, \hat y_n)}
{F_{\ell;t}^{B;t_0}|_\PP (y_1, \dots, y_n)} \, x_1^{g_1} \cdots
x_n^{g_n} \,,
\end{equation}
where $(g_1, \dots, g_n) = \g_{\ell;t}^{B;t_0}$, and the elements
$\hat y_1, \dots, \hat y_n \in \mathcal F$ are given by
\begin{equation}
\label{eq:yhat-0}
\hat y_j = y_j \prod_i x_i^{b_{i,j}}.
\end{equation}
Furthermore, all cluster variables are of this form, and $\mathcal
A$ is the $\Z \PP$-subalgebra of $\mathcal F$ generated by all the
$x_{\ell;t}$.

\section{$F$-polynomials of quiver representations}
\label{sec:F-QP}

In this section we use the terminology on quiver representations
from the introduction.
We work with a quiver $Q = Q(B)$ (see \eqref{eq:QB}).
Our goal is to develop some basic properties of the $F$-polynomial
$F_M(u_1, \dots, u_n)$ associated to any representation~$M$ of~$Q$
in accordance with \eqref{eq:F-M}.

\begin{proposition}
\label{pr:F-ct-ht}
Each polynomial $F_M (u_1, \dots, u_n)$ has constant term~$1$.
Furthermore, $F_M (u_1, \dots, u_n)$ contains the monomial
$\prod_{i=1}^n u_i^{\dim M(i)}$ with coefficient~$1$, and it is
divisible by all the other occurring monomials.
\end{proposition}

\begin{proof}
It is enough to notice that, for $\e = (0, \dots, 0)$ or $\e = \dd_M$
(see \eqref{eq:dim-vector}),
the quiver Grassmannian $\Gr_\e(M)$ consists of one point.
\end{proof}

\begin{proposition}
\label{pr:F-direct-sums}
For all representations $M'$ and $M''$ of~$Q$, we have
\begin{equation}
\label{eq:F-direct-sums}
F_{M'\oplus M''} = F_{M'} F_{M''}.
\end{equation}
\end{proposition}

\begin{proof}
We use the following well-known property of the Euler-Poincar\'e characteristic:
if a complex torus~$T$ acts algebraically on a variety~$X$, then
$\chi(X) = \chi(X^T)$, where $X^T$ is the set of $T$-fixed points (see for example \cite{bb}).
Take $X = \Gr_\e (M' \oplus M'')$, and consider the action of $T = \C^*$
on~$X$ induced by the $T$-action on $M' \oplus M''$ given by
$$t \cdot (m', m'') = (tm', m'') \quad (m' \in M', \,\, m'' \in M'').$$
Then a point $N \in X$ is $T$-fixed if and only if the submodule
$N \subseteq M' \oplus M''$ splits into $N = N' \oplus N''$ for
some $N' \subseteq M'$ and $N'' \subseteq M''$.
Thus, we have
$$X^T = \bigsqcup_{\e' + \e'' = \e} (\Gr_{\e'} (M') \times
\Gr_{\e''} (M'')),$$
and so
$$\chi(\Gr_\e (M' \oplus M'')) = \sum_{\e' + \e'' = \e}
\chi(\Gr_{\e'} (M'))\chi(\Gr_{\e''} (M'')),$$
implying \eqref{eq:F-direct-sums}.
\end{proof}

To state our next result, we recall the maps $\alpha_k$ and $\beta_k$ in
\eqref{eq:in-out-maps}
(the map $\gamma_k$ is undefined for arbitrary quiver
representations).
We will denote these maps $\alpha_{k;M}$ and $\beta_{k;M}$ if
necessary to stress the dependency of a representation~$M$.
We denote by $\h_{M} = (h_1, \dots, h_n)$ the integer
vector given by
\begin{equation}
\label{eq:h-ker-beta}
h_k  = h_k(M) = - \dim \ker \beta_{k;M}.
\end{equation}

Now assume that $F_M$ belongs to
$\Qsf(u_1, \dots, u_n)$ (the semifield of subtraction-free rational expressions),
hence can be evaluated in an arbitrary semifield \footnote{It is conceivable that this condition holds for
arbitrary quiver representations.
}(see the
discussion after Proposition~\ref{pr:F-poly}).
The definition \eqref{eq:h-ker-beta} is then justified by the
following analog of \eqref{eq:h}.

\begin{proposition}
\label{pr:h-ker-beta}
Under the assumption that $F_M \in \Qsf(u_1, \dots, u_n)$,
the components of the vector $\h_{M}$ appear as the exponents in
the tropical evaluation
\begin{equation}
\label{eq:h-qp}
x_1^{h_1} \cdots x_n^{h_n} = F_{M}|_{\Trop(x_1, \dots, x_n)}
(x_i^{-1} \prod_{j \neq i} x_j^{[-b_{j,i}]_+} \leftarrow u_i).
\end{equation}
\end{proposition}

\begin{proof}
First a general lemma following easily from the definition
of a tropical semifield (see \eqref{eq:tropical-addition}).

\begin{lemma}
\label{lem:tropical-newton}
If $F(u_1, \dots, u_n)$ is a Laurent polynomial belonging to
$\Qsf(u_1, \dots, u_n)$ then the result of any evaluation of~$F$
in a tropical semifield does not change if we replace~$F$ with the
sum of the terms (taken with coefficient~$1$) corresponding to the
vertices of its Newton polytope.
\end{lemma}

Now suppose that $N \in \Gr_\e(M)$, i.e., $N$ is a
subrepresentation of~$M$ with $\dd_N = \e$.
Then the exponent of~$x_k$ in the tropical evaluation
$$(u_1^{e_1} \cdots u_n^{e_n})|_{\Trop(x_1, \dots, x_n)}
(x_i^{-1} \prod_{j \neq i} x_j^{[-b_{j,i}]_+} \leftarrow u_i)$$
can be rewritten as
$$-e_k + \sum_{i \neq k} [-b_{k,i}]_+ e_i
= -e_k + \sum_{i \neq k} [b_{i,k}]_+ e_i = \dim N_{\rm out}(k) -
\dim N(k)$$
(we used the fact that $B$ is skew-symmetric).
Note that
\begin{align*}
\dim N_{\rm out}(k) - \dim N(k) & \geq \dim \beta(N(k)) - \dim N(k)\\
& = - \dim (N(k) \ \cap \ \ker \beta_k)\\
& \geq - \dim \ker \beta_k = h_k.
\end{align*}
In view of Lemma~\ref{lem:tropical-newton}, this implies that~$h_k$
does not exceed the exponent of~$x_k$ in the right hand side of
\eqref{eq:h-qp}.

Now take $\e = - h_k \e_k$ (recall that
$\e_k$ stands for the $k$-th unit vector in $\Z^n$), and
notice that $\Gr_\e(M)$ consists of one point $N$ (with $N(i) =
\{0\}$ for $i \neq k$, and $N(k) = \ker \beta_k$), and that $\e$ is
obviously a vertex of the Newton polytope of $F_M$.
This implies that the exponent of~$x_k$ in the right hand side of
\eqref{eq:h-qp} does not exceed $h_k$,
completing the proof of Proposition~\ref{pr:h-ker-beta}.
\end{proof}
\begin{proposition}\label{pr:smoothness-for-general-reps}
Suppose that $Q$ is a quiver with no oriented cycles, and $M$
is a general representation of dimension  ${\bf d}=(d_1,\dots,d_n)$.
Then every quiver Grassmannian $\Gr_{\bf e}(M)$ is smooth.
In particular, this is the case if $M$ is rigid,  that is, $\Ext^1(M,M)=0$.
\end{proposition}
This proposition follows from the results of Schofield (\cite[\S3]{schofield}).
For the convenience of the reader we give an outline of the proof.
\begin{proof}
If $M$ is a representation with dimension vector ${\bf d}$,
then we may identify $M(i)$ with $\C^{d_i}$ by choosing a basis in $M(i)$ for all $i$.
Then $M$ is represented as an element
$$
(a_M)_{a\in Q_1}\in \Rep_{\bf d}(Q):=\prod_{a\in Q_1} K^{d_{ha}\times d_{ta}}.
$$
The group $\GL_{\bf d}=\prod_{i=1}^n \GL_{d_i}(\C)$ acts on $\Rep_{\bf d}(Q)$
by base change. This way, isomorphism classes of ${\bf d}$-dimensional representations
correspond to $\GL_{\bf d}$-orbits in $\Rep_{\bf d}(Q)$.
For a dimension vector ${\bf e}=(e_1,\dots,e_n)$, let
$$Z_{{\bf e},{\bf d}}\subseteq \Rep_{{\bf d}}(Q)\times \prod_{i=1}^n \Gr_{e_i}(\C^{d_i})$$
be defined as the set of all
$(M,(N_1,N_2,\dots,N_n))$
for which $a_M(N_{ta})\subseteq N_{ha}$ for all  $a\in Q_1$.
We have natural projections $p:Z_{{\bf e},{\bf d}}\to \Rep_{\bf d}(Q)$
and $q:Z_{{\bf e},{\bf d}}\to \prod_{i=1}^n \Gr_{e_i}(\C^{d_i})$.
One can show that the projection $q$ makes $Z_{{\bf e},{\bf d}}$ into
a vector bundle over the product of Grassmannians, hence $Z_{{\bf
e},{\bf d}}$ is smooth.
Now the quiver Grassmannian $\Gr_{\bf e}(M)$ is
equal to the fiber $p^{-1}(M)$. If $M$ is a general representation of dimension ${\bf d}$,
then the fiber  $p^{-1}(M)$ is smooth by
 the second Bertini Theorem  (\cite[Chapter II, \S6.2, Theorem 2]{Shafarevich}).
\end{proof}

If $Q$ is a quiver without oriented cycles,
and $M$ is indecomposable and rigid, then all
the quiver Grassmannians are smooth by the proposition above.
It was shown in \cite{ck,cr}\footnote{It was pointed out in \cite{nakajima}
that the proof in \cite{cr} contains a gap.}, that the
$F$-polynomial of~$M$ has nonnegative coefficients in this case.
The next two examples show that in general, the coefficients can be negative,
and the quiver Grassmannian may be singular.

\begin{example}
Consider the quiver $Q$ given by
$$
\xymatrix{
1\ar@<.9ex>[rr]^{a_1,a_2,a_3,a_4}
\ar@<.3ex>[rr]
\ar@<-.3ex>[rr]
\ar@<-.9ex>[rr] & & 2}
$$
and let $M$ be a general representation of $Q$ of dimension $\dd = (3,4)$.
The arrows $a_{1},\dots,a_4$ act in~$M$ as four linear maps $\C^3 \to \C^4$ in general position.
Choose $\e= (1,3)$.
Since~$M$ is in general position, $\Gr_\e(M)$
is smooth by the discussion above. 
Now the first projection $\Gr_\e(M) \to \Gr_1(\C^3) = \PPP^2$
identifies $\Gr_\e(M)$ with the projective curve $C$ given by the equation
$$
\det(a_1 (m), \ a_2 (m),\ a_3 (m),  \ a_4 (m))=0 \quad (m \in \C^3).
$$
Since~$C$ is a smooth curve of degree~$4$, it has genus $g=(4-1)(4-2)/2=3$ and Euler characteristic
$2-2g=2-2\cdot 3=-4$ (see \cite[Chapter IV, 2.3]{Shafarevich}). So we have
$$
\chi(\Gr_\e(M))=-4.
$$
\end{example}

\begin{example}
\label{ex:nonsmooth}
Consider the quiver $Q$ given by
$$
\xymatrix{
& 1\ar^a[rd] &\\
3\ar^c[ru] & & 2\ar^b[ll]
}
$$
Let ${M}_1,{M}_2, {M}_3$ be the
indecomposable representations of $Q$ of dimensions
$(0,1,1)$, $(1,0,1)$, and $(1,1,0)$,
respectively, and $M ={M}_1\oplus {M}_2\oplus {M}_3$.
It is immediate from the definition \eqref{eq:F-M} that
\begin{align*}
& F_{{M}_1}(u_1,u_2,u_3)=1+u_3+u_2u_3,\\
& F_{{M}_2}(u_1,u_2,u_3)=1+u_1+u_1u_3,\\
& F_{{M}_3}(u_1,u_2,u_3)=1+u_2+u_1u_2.
\end{align*}
By Proposition~\ref{pr:F-direct-sums}, we have
$
F_{M}=F_{{M}_1}F_{{M}_2}F_{{M}_3}
$.
In particular, the coefficient of $u_1u_2u_3$ in $F_{{M}}$ is $4$.
Thus, $\chi(\Gr_{(1,1,1)}(M)) = 4$.

Geometrically this result can be seen as follows.
The variety $\Gr_{(1,1,1)}(M)$ is a subvariety in
$\Gr_1 (M(1)) \times \Gr_1 (M(2)) \times \Gr_1 (M(3)) =
\PPP^1 \times \PPP^1 \times \PPP^1$.
Let $P = (P(1), P(2), P(3)) \in \PPP^1 \times \PPP^1 \times \PPP^1$
be given by
$$P(1) = \ker a_M = {\rm im} \ c_M, \,\,
P(2) = \ker b_M = {\rm im} \ a_M, \,\,
P(3) = \ker c_M = {\rm im} \ b_M \ ;$$
then $\Gr_{(1,1,1)}(M)$ consists of all points
$N \in \PPP^1 \times \PPP^1 \times \PPP^1$
such that $N$ and $P$ have at least two common components.
Thus, $\Gr_{(1,1,1)}(M)$ is the union of three copies of $\PPP^1$
meeting at a single point~$P$.
In other words, $\Gr_{(1,1,1)}(M)$ is the disjoint union of three
copies of $\Affine^1$ and the single point $\{P\}$, so
$$
\chi(\Gr_e(N))=3\chi(\Affine^1)+\chi(\{P\})=3\cdot 1+1=4.
$$
Note that $\Gr_{(1,1,1)}(M)$ is singular at $P$.
\end{example}

\section{Background on quivers with potentials and their representations}
\label{sec:QP-background}

Let $Q = (Q_0, Q_1, h, t)$ be a quiver (see Introduction).
We denote by~$R$ the \emph{vertex span} of~$Q$, that is, the
commutative algebra over $\C$ with the basis $\{e_i: i \in Q_0\}$
and the multiplication given by $e_i e_j = \delta_{i,j} e_i$.
The \emph{arrow span} of~$Q$ is the finite-dimensional
$R$-bimodule~$A$ with the $\C$-basis identified with $Q_1$,
and the $R$-bimodule structure given by
\begin{equation}
\label{eq:A-R-bimodule}
A_{i,j} = e_i A e_j = \bigoplus_{a:j \to i} \C a.
\end{equation}
The \emph{complete path algebra} of~$Q$ is defined as
$$
R\langle\langle A \rangle\rangle=\prod_{d=0}^\infty A^{\otimes_R d}.
$$
Thus, the elements of $R\langle\langle A \rangle\rangle$ are
(possibly infinite) $\C$-linear combinations of paths in~$Q$; note
that by the convention \eqref{eq:A-R-bimodule} all the paths
are traced in the right-to-left order.
We view $R\langle\langle A \rangle\rangle$ as a topological
algebra with respect to the ${\mathfrak m}$-adic topology, where
the (two-sided) ideal ${\mathfrak m} \subset R\langle\langle A\rangle\rangle$
is given by
\begin{equation}
\label{eq:max-ideal}
{\mathfrak m}  = \prod_{d=1}^\infty A^{\otimes_R d}.
\end{equation}

A \emph{potential} on~$Q$ is an element
$S \in {\mathfrak m}_{\rm cyc} = \bigoplus_{i \in Q_0}
{\mathfrak m}_{i,i}$, i.e., a possibly infinite linear combination
of cyclic paths in $R\langle\langle A \rangle\rangle$.
We view potentials up to cyclical equivalence defined as follows:
two potentials $S$ and $S'$ are \emph{cyclically equivalent}
if $S - S'$ lies in the closure of the span of
all elements of the form $a_1 \cdots a_d - a_2 \cdots a_d a_1$,
where $a_1 \cdots a_d$ is a cyclic path.

For any arrow~$a \in Q_1$, the \emph{cyclic
derivative} $\partial_a$ is the continuous linear map
${\mathfrak m}_{\rm cyc} \to R\langle\langle A\rangle\rangle_{t(a),h(a)}$
acting on cyclic paths by
\begin{equation}
\label{eq:cyclic-derivative}
\partial_a (a_1 \cdots a_d) =
\sum_{p: a_p = a} a_{p+1} \cdots a_d a_1 \cdots a_{p-1}.
\end{equation}
The \emph{Jacobian ideal} $J(S)$ of a potential~$S$ is the closure of
the (two-sided) ideal in $R\langle\langle A\rangle\rangle$
generated by the elements $\partial_a (S)$ for all $a \in Q_1$.
We call the quotient $R\langle\langle A\rangle\rangle/J(S)$
the \emph{Jacobian algebra} of~$S$, and denote
it by ${\mathcal P}(Q,S)$ or ${\mathcal P}(A,S)$.

The cyclic derivatives of a potential~$S$ can be expressed in
terms of another important family of elements
$\partial_{ba}(S) \in R\langle\langle A\rangle\rangle$
associated with pairs of arrows $a, b \in Q_1$ such that $h(a) = t(b)$.
Namely, the definition of a continuous linear map
$$
\partial_{ba}: {\mathfrak m}_{\rm cyc}\to R\langle\langle
A\rangle\rangle_{t(a),h(b)}.
$$
is similar to \eqref{eq:cyclic-derivative}:
replacing if necessary a potential~$S$ with a
cyclically equivalent one, we can assume that
no cyclic path occurring in~$S$ starts with an arrow~$a$;
for every such cyclic path
$a_1 \cdots a_d$, we set
\begin{equation}
\label{eq:partial-ba}
\partial_{ba} (a_1 \cdots a_d) =
\sum_{\nu: a_{\nu-1} = b, a_\nu = a} a_{\nu+1} \cdots a_d a_1 \cdots a_{\nu-2}.
\end{equation}
An easy check shows that, for any $b \in Q_1$, we have
\begin{equation}
\label{eq:partial-ba-products}
\sum_{a: h(a) = t(b)} a \cdot \partial_{ba}(S) =
\sum_{c: t(c) = h(b)} \partial_{cb}(S) \cdot c
= \partial_b (S).
\end{equation}

A (decorated) representation of a quiver with potential $(Q,S)$
(QP for short) is a pair $\M = (M,V)$, where $M$ is a
finite-dimensional ${\mathcal P}(Q,S)$-module, and $V$ is
a finite-dimensional $R$-module.
A more concrete description was given in the introduction
(see \cite[Section~10]{dwz}): $V$ is simply a collection
$(V(i))_{i \in Q_0}$ of finite-dimensional vector spaces, while
$M = (M(i))_{i \in Q_0}$ is a representation of~$Q$
annihilated by ${\mathfrak m}^N$ for $N \gg 0$, and by
all cyclic derivatives of~$S$.
In view of \eqref{eq:partial-ba-products}, the latter relations
are equivalent to \eqref{eq:triangle-relations}, where the map
$\gamma_k$ in the triangle \eqref{eq:triangle} is defined as
follows: for each $a, b \in Q_1$ with $h(a) = t(b) = k$, the
component $\gamma_{a,b}: M(h(b)) \to M(t(a))$ of $\gamma_k$
is given by \eqref{eq:gamma-ab}.
We can also express $\gamma_k$ in  matrix form: set
\begin{equation}
\label{eq:arrows-thru-k}
\{a_1, \dots, a_r\} = \{a \in Q_1: h(a) = k\}, \quad
\{b_1, \dots, b_s\} = \{b \in Q_1: t(b) = k\},
\end{equation}
and let $H_k(S)$ be the $r \times s$ matrix whose $(p,q)$ entry is
$\partial_{b_q a_p} S$; then the action of $\gamma_k$ in~$M$ is
given by the matrix
\begin{equation}
\label{eq:gamma-matrix}
\gamma_k = (H_k(S))_M.
\end{equation}

In what follows, we refer to a decorated representation $\M = (M,V)$ of a QP
$(Q,S)$ as a \emph{QP-representation}.
The direct sums and indecomposable QP-representations are defined
in a natural way.
We say that $\M$ is \emph{positive} if $V = \{0\}$, and
\emph{negative} if $M = \{0\}$.
Thus, indecomposable positive QP-representations are just
indecomposable ${\mathcal P}(Q,S)$-modules, while
indecomposable negative QP-representations are
\emph{negative simple} representations ${\mathcal S}_k^-$ for $k
\in Q_0$ defined as follows:
\begin{equation}
\label{eq:neg-simple-def}
{\mathcal S}_k^-(Q,S) = (\{0\},V), \,\, \dim V(i) = \delta_{i,k}.
\end{equation}

As in \cite[Definitions~4.2, 10.2]{dwz}, we view QP's and their representations
up to \emph{right-equivalence}.
Recall that QP's $(Q,S)$ and $(Q,S')$ on the same underlying
quiver~$Q$ are right-equivalent if there is an
automorphism~$\varphi$ of $R\langle\langle A\rangle\rangle$ (as an
algebra and $R$-bimodule) such that $\varphi(S)$ is cyclically equivalent to~$S'$.
In view of \cite[Proposition~3.7]{dwz}, we then have $\varphi(J(S)) =
J(S')$; therefore, every ${\mathcal P}(Q,S)$-module~$M$ carries a
structure of a ${\mathcal P}(Q,S')$-module (which we denote ${}^\varphi M$)
with the ``twisted" action of $R\langle\langle
A\rangle\rangle$ given by
$$\varphi(u) \star m = u m \quad (u \in R\langle\langle A \rangle\rangle, \,\, m \in M).$$
Now a QP-representation $\M' = (M',V')$ of $(Q,S')$ is
\emph{right-equivalent} to a QP-representation $\M = (M,V)$ of
$(Q,S)$ if $M'$ is isomorphic to ${}^\varphi M$ as a ${\mathcal P}(Q,S')$-module,
and $V'$ is isomorphic to $V$ as an $R$-module.

Let $\varphi$ be an automorphism of $R\langle\langle A\rangle\rangle$ as above.
Fix a vertex $k \in Q_0$, and use the notation in \eqref{eq:arrows-thru-k}.
We would like to express the matrix $H_k (\varphi(S))$ in terms of $H_k (S)$.
As shown in the proof of Lemma~5.3 in \cite{dwz}, we have
\begin{equation}
\label{eq:phi-on-a}
\begin{pmatrix}
\varphi(a_1) & \varphi(a_2) & \cdots & \varphi(a_r)\end{pmatrix}
=
\begin{pmatrix}
a_1 & a_2 & \cdots & a_r\end{pmatrix}
(C_0 + C_1),
\end{equation}
where:
\begin{itemize}
\item $C_0$ is an invertible $r \times r$ matrix with entries
in~$\C$ such that its $(p,q)$-entry is~$0$ unless $t(a_p) = t(a_q)$;
\item $C_1$ is a $r \times r$ matrix whose $(p,q)$-entry belongs
to ${\mathfrak m}_{t(a_p),t(a_q)}$.
\end{itemize}
Similarly, we have
\begin{equation}
\label{eq:phi-on-b}
\begin{pmatrix}
\varphi(b_1)\\ \varphi(b_2)\\ \vdots \\ \varphi(b_s)\end{pmatrix}
=
(D_0 + D_1)\begin{pmatrix}
b_1 \\ b_2 \\ \vdots \\ b_s\end{pmatrix}
,
\end{equation}
where:
\begin{itemize}
\item $D_0$ is an invertible $s \times s$ matrix with entries
in~$\C$ such that its $(p,q)$-entry is~$0$ unless $h(b_p) = h(b_q)$;
\item $D_1$ is a $s \times s$ matrix whose $(p,q)$-entry belongs
to ${\mathfrak m}_{h(b_p),h(b_q)}$.
\end{itemize}
Note that both matrices $C_0 + C_1$ and $D_0 + D_1$ are invertible, and their inverses are of the
same form.

In the above notation, we claim that
\begin{align}
\label{Hk-phiS-S}
&\text{all entries of the matrix}\\
\nonumber
&\text{$H_k (\varphi(S)) - (C_0 + C_1)\, \varphi(H_k (S))\, (D_0 + D_1)$ belong to
$J(\varphi(S))$}
\end{align}
(here the matrix $\varphi(H_k (S))$ is obtained by
applying~$\varphi$ to each entry of $H_k (S)$).
As a consequence, for the representation $M' = {}^\varphi M$ as
above, the corresponding map $\gamma'_k$ is given by
\begin{equation}
\label{eq:gamma-gamma'}
\gamma'_k = (C_0 + C_1)_{M'}\circ \gamma_k \circ (D_0 + D_1)_{M'},
\end{equation}
where $(C_0 + C_1)_{M'}$ (resp. $(D_0 + D_1)_{M'}$) is an $R$-bimodule
automorphism of $M_{\rm in}(k)$ (resp. of $M_{\rm out}(k)$).
Note that \eqref{eq:gamma-gamma'} is the equality (10.16) in
\cite{dwz}, while \eqref{Hk-phiS-S} is implicit in the proof of this equality.

We now recall one of the main technical results of \cite{dwz}, the
Splitting Theorem (\cite[Theorem~4.6]{dwz}).
Let $Q$ be a quiver without loops (but possibly having oriented
$2$-cycles).
We say that a QP $(Q,S)$ is \emph{trivial} if $S$ is a
linear combination of cyclic $2$-paths, and $J(S) = {\mathfrak m}$;
in other words (see \cite[Proposition~4.4]{dwz}),
the set of arrows $Q_1$ consists of $2N$ distinct arrows
$a_1, b_1, \dots, a_N, b_N$ such that
each $a_\nu b_\nu$ is a cyclic $2$-path, and
there is an $R$-bimodule automorphism~$\varphi$ of the arrow
span~$A$ such that $\varphi(S)$ is cyclically equivalent to $a_1 b_1 + \cdots + a_N b_N$.
We say that a QP $(Q,S)$ is \emph{reduced} if $S \in {\mathfrak m}^3$
(note that $Q$ is still allowed to have oriented $2$-cycles).
Now the Splitting Theorem asserts that
\begin{equation}\label{eq:splitting-theorem}
\vcenter{\noindent\hsize=5in any QP $(Q,S)$ is right-equivalent to the direct sum of
a reduced QP $(Q,S)_{\rm red}$ and a trivial QP $(Q,S)_{\rm triv}$, each
of which is determined by $(Q,S)$ up to right-equivalence.}
\end{equation}


We refer to $(Q,S)_{\rm red}$ as the \emph{reduced part} of $(Q,S)$.
The operation of taking the reduced part naturally extends to
representations.
Namely, if $\M = (M,V)$ is a representation of $(Q,S)$, then
$\M_{\rm red}$ is obtained by transforming $\M$ into a representation
$({}^\varphi M, V)$ of $(Q,S)_{\rm red} \oplus (Q,S)_{\rm triv}$ with the help
of a right-equivalence in \eqref{eq:splitting-theorem}, and then
restricting the resulting representation to $(Q,S)_{\rm red}$
(see \cite[Definition~10.4]{dwz} for more details).
By \cite[Proposition~10.5]{dwz}, the reduction of representations
is well-defined on the level of right-equivalence classes.

Now everything is in place for introducing our main tool --
mutations of reduced QP's and their representations.
Let $(Q,S)$ be a reduced QP, and $k \in Q_0$ a vertex such that $Q$ has no
oriented $2$-cycles through~$k$.
Following \cite{dwz}, we define the \emph{mutation}
$(\overline Q, \overline S) = \mu_k(Q,S)$ at~$k$ as the reduced
part $(\widetilde Q, \widetilde S)_{\rm red}$, where the
``premutation" $(\widetilde Q, \widetilde S) = \widetilde \mu_k (Q,S)$
is defined as follows.
First, the quiver $\widetilde Q$ is obtained from
$Q$ by the following two-step procedure:
\begin{enumerate}
\item[{\bf Step 1.}] For every pair of arrows $a, b \in Q_1$ with
$h(a) = k = t(b)$, create a ``composite" arrow $[ba]$ with
$h([ba]) = h(b)$ and $t([ba]) = t(a)$.
\item[{\bf Step 2.}] Reverse all arrows at~$k$; that is, replace each arrow
$a$ with $h(a) = k$ (resp. each arrow $b$ with $t(b) = k$) by
an arrow $a^\star$ with $t(a^\star) = k$ and $h(a^\star) = t(a)$
(resp. $b^\star$ with $h(b^\star) = k$ and $t(b^\star) = h(b)$).
\end{enumerate}
Second, the potential $\widetilde S$ on $\widetilde Q$ is obtained
from $S$ as follows: replacing~$S$ if necessary with a
cyclically equivalent potential, we can assume that
no cyclic path occurring in~$S$ starts and ends at~$k$;
then we set
\begin{equation}
\label{eq:mu-k-S}
\widetilde S = [S] + \Delta,
\end{equation}
where
\begin{equation}
\label{eq:Tk}
\Delta = \sum_{a, b \in Q_1: \ h(a) = t(b) = k} [ba]a^\star b^\star,
\end{equation}
and $[S]$ is obtained by substituting $[a_\nu a_{\nu+1}]$ for each
factor $a_\nu a_{\nu+1}$ with $t(a_\nu) = h(a_{\nu+1}) = k$ of any cyclic path
$a_1 \cdots a_d$ occurring in the expansion of~$S$.
As shown in \cite[Theorem~5.2]{dwz}, the right-equivalence class
of $\widetilde \mu_k (Q,S)$ is determined by the
right-equivalence class of
$(Q,S)$; hence by \eqref{eq:splitting-theorem}, the same is true
for $\mu_k (Q,S) = (\widetilde \mu_k (Q,S))_{\rm red}$.
Furthermore, by \cite[Theorem~5.7]{dwz}, the mutation $\mu_k$
acts as an involution on the set of right-equivalence classes of reduced QPs, that is,
$\mu_k^2(Q,S)$ is right-equivalent to $(Q,S)$.

Now let $\M = (M,V)$ be a QP-representation of a reduced QP $(Q,S)$.
Fix a vertex~$k$ and let $(\widetilde Q, \widetilde S) = \widetilde \mu_k (Q,S)$, and
$(\overline Q, \overline S) = \mu_k(Q,S) = (\widetilde Q, \widetilde S)_{\rm red}$.
We define the mutated QP-representation $\overline \M = \mu_k(\M)$ of $(\overline Q, \overline S)$
as the reduced part of the QP-representation
$\widetilde \M = \widetilde \mu_k(\M) = (\overline M, \overline
V)$ of $(\widetilde Q, \widetilde S)$ given by the
following construction (see \cite[Section~10]{dwz}).

First, we set
\begin{equation}
\label{eq:M-unchanged}
\overline M(i) = M(i), \quad  \overline V(i) = V(i) \quad (i \neq k),
\end{equation}
and define the spaces $\overline M(k)$ and $\overline V(k)$ by
\begin{equation}
\label{eq:new-Mk}
\overline M(k) = \frac{\ker \gamma_k}{{\rm im}\ \beta_k} \oplus
{\rm im}\ \gamma_k \oplus \frac{\ker \alpha_k}{{\rm im}\ \gamma_k}
\oplus V(k), \quad  \overline V(k) =
\frac{\ker \beta_k}{\ker \beta_k \cap {\rm im}\ \alpha_k}\
\end{equation}
(see \eqref{eq:triangle}).
For every arrow~$c$ of $\widetilde Q$, the corresponding
linear map $c_{\overline M}: \overline M(t(c)) \to \overline M(h(c))$
is defined as follows.

We set $c_{\overline M} = c_{M}$
for every arrow~$c$ not incident to~$k$, and
$[b a]_{\overline M} =  b_M a_M$
for all arrows $a$ and $b$ in $Q$ with $h(a) = k = t(b)$.
It remains to define the linear maps
$$\overline \alpha_k: \overline M_{\rm in}(k) = M_{\rm out}(k) \to
\overline M(k), \quad
\overline \beta_k: \overline M(k) \to \overline M_{\rm out}(k) = M_{\rm in}(k)$$
in the counterpart of the triangle \eqref{eq:triangle} for the
representation $\widetilde \M$.
We use the following notational convention: whenever we have
a pair $U_1 \subseteq U_2$ of vector spaces,
denote by $\iota:U_1 \to U_2$ the inclusion map, and by $\pi: U_2 \to U_2/U_1$
the natural projection.
We now introduce the following \emph{splitting data}:
\begin{equation}
\label{eq:rho}
\vbox{\hsize=5in\noindent
Choose a linear map $\rho: M_{\rm out}(k) \to \ker \gamma_k$
such that $\rho \iota = {\rm id}_{\ker \gamma_k}$.}
\end{equation}
\begin{equation}\label{eq:sigma}\vbox{\hsize=5in\noindent
Choose a linear map $\sigma: \ker \alpha_k / {\rm im}\ \gamma_k
\to \ker \alpha_k$ such that $\pi \sigma = {\rm id}_{\ker \alpha_k/ {\rm im} \gamma_k}$.}
\end{equation}
Then we define:
\begin{equation}
\label{eq:alpha-beta-mutated}
\overline \alpha_k =
\begin{pmatrix}
- \pi \rho\\
- \gamma_k\\
0\\
0
\end{pmatrix}, \quad
\overline \beta_k  =
\begin{pmatrix}
0 & \iota & \iota \sigma & 0\end{pmatrix}.
\end{equation}

As shown in \cite[Propositions~10.7, 10.9, 10.10]{dwz}, the above construction makes
$\widetilde \mu_k ({\mathcal M})=(\overline M, \overline V)$
a QP-representation of $(\widetilde Q, \widetilde S)$, whose isomorphism class
does not depend on the choice of the splitting data \eqref{eq:rho} --
\eqref{eq:sigma}, and whose right-equivalence class
is determined by the right-equivalence class of~${\mathcal M}$.
Furthermore, we have
\begin{equation}
\label{eq:overline-gamma}
\overline \gamma_k = \beta_k \alpha_k,
\end{equation}
and
\begin{eqnarray}
\label{eq:overline-kernels-images}
&\ker \overline \alpha_k = {\rm im}\ \beta_k,
\quad {\rm im}\ \overline \alpha_k =
\displaystyle\frac{\ker \gamma_k}{{\rm im}\ \beta_k} \oplus
{\rm im}\ \gamma_k \oplus \{0\}
\oplus \{0\},\\
\nonumber
&\ker \overline \beta_k =
\displaystyle \frac{\ker \gamma_k}{{\rm im}\ \beta_k} \oplus
\{0\} \oplus \{0\}
\oplus V(k),
\quad \,\,\,{\rm im}\ \overline \beta_k = \ker \alpha_k .
\end{eqnarray}
(see \cite[(10.25), (10.26)]{dwz}).

Since by the definition, the representation $\mu_k(\M)$ of $(\overline Q, \overline S)$
is the reduced part of $\widetilde \M = \widetilde \mu_k(\M) = (\overline M,
\overline V)$, the right-equivalence class of $\mu_k(\M)$ is
determined by the right-equiva\-lence class of $\M$.
Furthermore, in view of \cite[Theorem~10.13]{dwz},
the mutation $\mu_k$ of QP-representations is an
involution:
\begin{equation}
\label{eq:muk-rep-involution}
\vcenter{\hsize=5in\noindent
for every QP-representation $\mathcal M$ of a reduced QP $(Q,S)$,
the QP-re\-pre\-sen\-ta\-tion $\mu_k^2 (\mathcal M)$
is right-equivalent to $\mathcal M$.}
\end{equation}
Since by construction, the mutations send direct sums of QP-representations to the direct
sums, \eqref{eq:muk-rep-involution} implies that (cf.
\cite[Corollary~10.14]{dwz})
\begin{equation}
\label{eq:rep-mutations-respect-idecomposables}
\vcenter{\hsize=5in\noindent
any mutation $\mu_k$ sends indecomposable QP-representations of reduced QPs to indecomposable ones.}
\end{equation}

Now suppose that the quiver $Q$ has no oriented $2$-cycles, i.e.,
it is of the form $Q(B)$ for some skew-symmetric integer matrix~$B$
(see \eqref{eq:QB}).
Then the mutated QP $\mu_k(Q,S) = (\overline Q, \overline S)$ is
well-defined for any vertex~$k$ and any potential~$S$ on~$Q$.
However, the quiver $\overline Q$ may acquire some oriented
$2$-cycle, say involving vertices~$i$ and~$j$, which would make
mutations $\mu_i$ and $\mu_j$ undefined for the QP $(\overline Q, \overline S)$.
Following \cite[Definition~7.2]{dwz}, we say that a QP $(Q,S)$ is
\emph{nondegenerate} if this does not happen, and moreover if any
finite sequence of mutations $\mu_{k_\ell} \cdots \mu_{k_1}$ can be applied to $(Q,S)$
without creating oriented $2$-cycles along the way.
According to this definition, the class of nondegenerate QPs is
stable under all mutations.
Furthermore, according to \cite[Proposition~7.1]{dwz}, mutations of
nondegenerate QPs are compatible with matrix mutations: if
$\mu_k(Q(B),S) = (\overline Q, \overline S)$ then
$\overline Q = Q(\mu_k(B))$ with $\mu_k(B)$ given by
\eqref{eq:B-mutation}.

Finally we note that every quiver $Q(B)$ has a
potential~$S$ such that $(Q(B),S)$ is a nondegenerate QP.
More precisely, in view of \cite[Corollary~7.4]{dwz}, the non-degeneracy
of $(Q(B),S)$ is guaranteed by non-vanishing at~$S$ of countably many
nonzero polynomial functions on the space of potentials on~$Q(B)$
(taken up to cyclical equivalence).

\section{QP-interpretation of $\g$-vectors and $F$-polynomials}
\label{sec:qp-interpretation}

We retain all the notation and conventions of the preceding sections.
To a QP-represen\-ta\-tion $\M = (M,V)$ we associate
the $\g$-vector $\g_{\M} = (g_1, \dots, g_n) \in \Z^n$ given by
\eqref{eq:g-M}, and the $F$-polynomial $F_\M = F_M$ given by
\eqref{eq:F-M} (in particular, if $\M$ is negative then $F_\M = 1$).
Note that $\g_{\M} = \g_{\M'}$ and $F_{\M} = F_{\M'}$ if $\M$ and
$\M'$ are right-equivalent (for the $F$-polynomial, this is
immediate from \eqref{eq:F-M}; for the $\g$-vector, this is a
consequence of \eqref{eq:gamma-gamma'}).
Note also that
\begin{equation}
\label{eq:g-sum}
\g_{\M \oplus \M'} = \g_{\M} + \g_{\M'}
\end{equation}
for any QP-representations $\M$ and $\M'$ of the same QP.

Let $B$ be a
skew-symmetric integer $n \times n$ matrix, $t_0, t \in \TT_n$,
and $\ell \in \{1, \dots, n\}$.
Let $Q = Q(B)$ and let $S$ be a potential on~$Q$ such that $(Q,S)$
is a nondegenerate QP.
The main result of this section is a construction of a
QP-representation $\M = \M_{\ell;t}^{B;t_0}$ of $(Q,S)$ such that
$\g_\M = \g_{\ell;t}^{B;t_0}$ and $F_\M = F_{\ell;t}^{B;t_0}$,
where the $\g$-vectors $\g_{\ell;t}^{B;t_0}$ and $F$-polynomials $F_{\ell;t}^{B;t_0}$
were introduced in Section~\ref{sec:F-poly}.

The family of QP-representations $\M_{\ell;t}^{B;t_0}$
is uniquely determined by the properties \eqref{eq:neg-simple} and \eqref{eq:Mlt-mutation}.
More explicitly, let
$$t_0 \overunder{k_1}{} t_1 \overunder{k_2}{}  \cdots \overunder{k_p}{} t_p = t$$
be the (unique) path joining $t_0$ and $t$ in $\TT_n$.
We set
$$(Q(t),S(t)) = \mu_{k_p} \cdots \mu_{k_1} (Q,S),$$
which is well-defined because~$(Q,S)$ is nondegenerate.
Let ${\mathcal S}_\ell^-(Q(t),S(t))$ be the negative simple representation of
$(Q(t),S(t))$ at a vertex~$\ell$ (see \eqref{eq:neg-simple-def}).
Then we have
\begin{equation}
\label{eq:M-cluster-variable}
\M_{\ell;t}^{B;t_0} = \mu_{k_1} \cdots \mu_{k_p} ({\mathcal S}_\ell^-(Q(t),S(t)));
\end{equation}
in view of \eqref{eq:muk-rep-involution}, replacing
$\M_{\ell;t}^{B;t_0}$ if necessary  by a right-equivalent
representation, we can assume that it is a QP-representation of $(Q,S)$.

\begin{theorem}
\label{th:g-F-rep}
We have
\begin{equation}
\label{eq:g-F-rep}
\g_{\ell;t}^{B;t_0} = \g_\M, \quad F_{\ell;t}^{B;t_0} = F_\M,
\end{equation}
where $\M = \M_{\ell;t}^{B;t_0}$.
\end{theorem}

\begin{proof}
We deduce Theorem~\ref{th:g-F-rep} from the following key lemma.

\begin{lemma}
\label{lem:g-F-mutation}
Let $\M = (M,V)$ be an arbitrary QP-representation of a
nondegenerate QP $(Q(B),S)$, let $\overline \M = (\overline M,
\overline V) = \mu_k(\M)$ for some $k \in Q(B)_0$, and suppose that
the $Y$-seed $(\yy', B_1)$ in $\Qsf(y_1, \dots, y_n)$
is obtained from $(\yy, B)$ by the mutation at~$k$.
Let $h_k$ (resp.~$h'_k$) be the $k$-th component of the
vector $\h_M$ (resp. $\h_{\overline M}$) given by \eqref{eq:h-ker-beta}.
Then the $\g$-vector $\g_\M = (g_1, \dots, g_n)$ satisfies
\eqref{eq:g-thru-h}, and is related to the $\g$-vector
$\g_{\overline \M} = (g'_1, \dots, g'_n)$ via
\eqref{eq:g-transition}.
Furthermore, the $F$-polynomials $F_{\M}$ and $F_{\overline \M}$
are related by
\begin{equation}
\label{eq:F-M-Mbar}
(y_k + 1)^{h_k}F_{\M}(y_1, \dots, y_n)
= (y'_k + 1)^{h'_k}
F_{\overline \M}(y'_1, \dots, y'_n)\,.
\end{equation}
\end{lemma}

Before proving Lemma~\ref{lem:g-F-mutation}, we first show how it
implies Theorem~\ref{th:g-F-rep}.
Let $\M = (M,V) = \M_{\ell;t}^{B;t_0}$.
We prove \eqref{eq:g-F-rep} together with the equality
\begin{equation}
\label{eq:h-rep-cluster}
\h_M = \h_{\ell;t}^{B;t_0}
\end{equation}
(see \eqref{eq:h}) by induction on the
distance between $t_0$ and $t$ in the tree $\TT_n$.
The basis of induction is the case $t = t_0$.
By \eqref{eq:neg-simple}, we have
$\M_{\ell;t_0}^{B;t_0} = \mathcal S_\ell^-(Q(B),S)$.
The fact that the $\g$-vector and $F$-polynomial of this
QP-representation agree with \eqref{eq:gg-initial} and
\eqref{eq:F-initial}, is immediate from the definitions,
while both sides of \eqref{eq:h-rep-cluster} are equal to~$0$.

Now assume that \eqref{eq:g-F-rep} and \eqref{eq:h-rep-cluster}
are satisfied for some $\ell$ and $t$, and that $t_0 \overunder{k}{} t_1$ in $\TT_n$.
In view of \eqref{eq:Mlt-mutation}, the QP-representation
$\overline \M = (\overline M, \overline V)$ in Lemma~\ref{lem:g-F-mutation}
is equal to $\M_{\ell;t}^{B';t_1}$, where $B' = \mu_k(B)$.
To finish the proof, it suffices to show that
\begin{enumerate}
\item $\g_{\overline \M} = \g_{\ell;t}^{B_1;t_1}$;
\item $F_{\overline \M} = F_{\ell;t}^{B_1;t_1}$;
\item $\h_{\overline \M} = \h_{\ell;t}^{B_1;t_1}$.
\end{enumerate}

To prove (1),
it suffices to observe that, by Lemma~\ref{lem:g-F-mutation}, the
vector $\g_{\overline \M}$ is obtained from $\g_{\M}$ by the same
rule \eqref{eq:g-transition} that expresses
$\g_{\ell;t}^{B_1;t_1}$ in terms of $\g_{\ell;t}^{B;t_0}$.
Then, since by Lemma~\ref{lem:g-F-mutation} the numbers $h'_k$,
$h_k$ and $g_k$ are related by \eqref{eq:g-thru-h}, we conclude
that $h'_k$ is the $k$-th component of the vector
$\h_{\ell;t}^{B_1;t_1}$.
Next, using the latter claim, and comparing \eqref{eq:F-M-Mbar}
with the relation \eqref{eq:F-t0-t1} in
Proposition~\ref{pr:g-F-t0-t1}, we obtain the proof of (2).
Finally, to prove (3) it is enough to apply Proposition~\ref{pr:h-ker-beta}
to the representation $\overline M$ (note that in view of (2), the
polynomial $F_{\overline \M}$ is a subtraction-free rational
expression, which makes Proposition~\ref{pr:h-ker-beta} applicable).

It remains to prove Lemma~\ref{lem:g-F-mutation}, which we
accomplish in several steps.

\smallskip

\noindent {\bf Step 1.}
We start by proving that the numbers $h'_k$,
$h_k$ and $g_k$ in Lemma~\ref{lem:g-F-mutation}
are related by \eqref{eq:g-thru-h}, which we rewrite as
$$- h'_k = g_k - h_k.$$
Remembering \eqref{eq:h-ker-beta} and \eqref{eq:g-M}, we can
rewrite this equality as
$$\dim \ker \overline \beta_k = \dim \ker \gamma_k - \dim M(k) + \dim V(k)
+ \dim \ker \beta_k
= \dim\Big(\frac{\ker \gamma_k}{{\rm im}\ \beta_k}
\oplus V(k)\Big),$$
which is immediate from \eqref{eq:overline-kernels-images}.

\smallskip

\noindent {\bf Step 2.} Our next target is the identity \eqref{eq:F-M-Mbar}.
Suppose that
$N = (N(1), \dots, N(n)) \in \prod_{i = 1}^n \Gr_{e_i}(M(i))$, and
let $N_{\rm in}(k)$ and $N_{\rm out}(k)$
be the corresponding subspaces of $M_{\rm in}(k)$ and $M_{\rm out}(k)$,
respectively.
The condition that $N \in \Gr_\e(M)$ can be stated as the combination of the
following two conditions:
\begin{equation}
\label{eq:N-outside-k}
\text{$c_M (N(j)) \subseteq N(i)$ for any arrow $c: j \to i$ not
incident to~$k$ in $Q(B)$.}
\end{equation}
\begin{equation}
\label{eq:N-thru-k}
\alpha_k(N_{\rm in}(k)) \subseteq N(k) \subseteq \beta_k^{-1}(N_{\rm out}(k)).
\end{equation}
Now let $\e' = (e_i)_{i \neq k}$ denote the integer vector
obtained from $\e$ by forgetting the component~$e_k$.
For every such vector $\e'$ and every pair of nonnegative integers
$r \leq s$, we denote by $Z_{\e';r,s}(M)$ the variety of tuples
$(N(i))_{i \neq k}$ satisfying the inclusions \eqref{eq:N-outside-k} and
$\alpha_k(N_{\rm in}(k)) \subseteq \beta_k^{-1}(N_{\rm out}(k))$,
and such that $\dim N(i) = e_i$ for $i \neq k$, and
$$\dim \alpha_k(N_{\rm in}(k)) = r, \quad \dim \beta_k^{-1}(N_{\rm out}(k))= s.$$
Let $\tilde Z_{\e;r,s}(M)$ denote the subset of $\Gr_\e(M)$
consisting of all $N  = (N(1), \dots, N(n))$ such that the tuple
obtained from $N$ by forgetting $N(k)$ belongs to $Z_{\e';r,s}(M)$.
Then $\Gr_\e(M)$ is the disjoint union of the subsets $\tilde Z_{\e;r,s}(M)$
over all pairs $(r,s)$; and in view of \eqref{eq:N-thru-k}, each $\tilde Z_{\e;r,s}(M)$ is the fiber
bundle over $Z_{\e';r,s}(M)$ with the fiber $\Gr_{e_k-r}(\C^{s-r})$.
Since $\chi(\Gr_{e_k-r}(\C^{s-r})) = \binom{s-r}{e_k-r}$, it
follows that
$$\chi(\Gr_\e (M)) = \sum_{r,s} \binom{s-r}{e_k-r} \chi(Z_{\e';r,s}(M)).$$
Substituting this expression into \eqref{eq:F-M} and performing
the summation with respect to~$e_k$, we obtain
\begin{equation}
\label{eq:F-M-thru-k}
F_\M(y_1, \dots, y_n) =
\sum_{\e',r,s} \chi(Z_{\e';r,s}(M))y_k^r (y_k+1)^{s-r} \prod_{i \neq k}^n y_i^{e_i}.
\end{equation}

The proof of \eqref{eq:F-M-Mbar} is based on the following observation:
\begin{equation}
\label{eq:Z=barZ}
Z_{\e';r,s}(M) = Z_{\e';\overline r,\overline s}(\overline M),
\end{equation}
where $\overline r$ and $\overline s$ are given by
\begin{equation}
\label{eq:barr-bars}
\overline r = \sum_i [b_{i,k}]_+ e_i - h_k - s, \quad
\overline s = \sum_i [-b_{i,k}]_+ e_i - h'_k - r.
\end{equation}
In view of the symmetry between $\M$ and $\overline \M$, to prove \eqref{eq:Z=barZ},
it is enough to show that every $(N(i))_{i \neq k} \in Z_{\e';r,s}(M)$
belongs to $Z_{\e';\overline r,\overline s}(\overline M)$.

First of all, we need to show that
$\overline \beta_k \overline \alpha_k(N_{\rm out}(k)) \subseteq N_{\rm in}(k)$,
that is, the counterpart for $\overline M$ of the inclusion
$\alpha_k(N_{\rm in}(k)) \subseteq \beta_k^{-1}(N_{\rm out}(k))$.
As an immediate consequence of \eqref{eq:alpha-beta-mutated}, we get
$\overline \beta_k \overline \alpha_k = - \gamma_k$.
In view of \eqref{eq:gamma-ab}, each of the components of the map
$\gamma_k$ is a linear combination of compositions of maps of the
kind $c_M$ or $b_M a_M$ (where $a, b, c \in Q_1$ are such that
$h(a) = t(b) = k$, and $c$ is not incident to~$k$); thus, the
defining conditions \eqref{eq:N-outside-k} and \eqref{eq:N-thru-k}
imply the desired inclusion $\gamma_k(N_{\rm out}(k)) \subseteq N_{\rm in}(k)$.

To conclude the proof of \eqref{eq:Z=barZ}, it remains to show that
\begin{equation}
\label{eq:rbar-sbar-dimensions}
\dim \overline \alpha_k(N_{\rm out}(k)) = \overline r, \quad
\dim \overline \beta_k^{-1}(N_{\rm in}(k))= \overline s.
\end{equation}
To show the first equality, recall from \eqref{eq:overline-kernels-images} that
$\ker \overline \alpha_k = {\rm im} \ \beta_k$,
implying that
$$\dim \overline \alpha_k(N_{\rm out}(k)) =
\dim N_{\rm out}(k)/(N_{\rm out}(k) \cap {\rm im} \ \beta_k) =
\sum_i [b_{i,k}]_+ e_i - \dim (N_{\rm out}(k) \cap {\rm im} \ \beta_k).$$
Using the exact sequence
$$0 \to \ker \beta_k \to \beta_k^{-1}(N_{\rm out}(k)) \to
N_{\rm out}(k) \cap {\rm im} \ \beta_k \to 0,$$
we conclude that
$$\dim (N_{\rm out}(k) \cap {\rm im} \ \beta_k) = \dim \beta_k^{-1}(N_{\rm
out}(k)) - \dim \ker \beta_k = s + h_k,$$
implying the first equality in \eqref{eq:rbar-sbar-dimensions}.
The second equality can be shown by similar arguments but also
follows from the first one applied to $\overline \M$ instead of $\M$.

The rest of the proof of \eqref{eq:F-M-Mbar} is straightforward: use \eqref{eq:F-M-thru-k} and
\eqref{eq:Z=barZ} for rewriting its right-hand
side in the form
$$(y'_k + 1)^{h'_k}
F_{\overline \M}(y'_1, \dots, y'_n) =
(y'_k + 1)^{h'_k}
\sum_{\e',r,s} \chi(Z_{\e';r,s}(M))(y'_k)^{\overline r}
(y'_k+1)^{\overline s - \overline r} \prod_{i \neq k} (y'_i)^{e_i},$$
then substitute for $y'_1, \dots, y'_n$ (resp. for $\overline r$ and $\overline s$)
the expressions given by \eqref{eq:y-mutation} (resp. by \eqref{eq:barr-bars}), simplify the resulting
expression, and use \eqref{eq:F-M-thru-k} again to see that it is equal to
the left-hand-side of \eqref{eq:F-M-Mbar}.

\smallskip

\noindent {\bf Step 3.} To finish the proof of
Lemma~\ref{lem:g-F-mutation}, it remains to show that the vectors
$\g_\M$ and $\g_{\M'}$ are related by \eqref{eq:g-transition}.
As shown in Step~1, we have $g_k = h_k - h'_k$, implying
the equality $g'_k = -g_k$.

Now let $i \neq k$.
Using \eqref{eq:g-M}, \eqref{eq:h-ker-beta}, and the fact that the matrix $B$ is
skew-symmetric, we can rewrite the desired second equality in \eqref{eq:g-transition} as
$$\dim \ker \gamma_i - [b_{k,i}]_+ \dim \ker \beta_k =
\dim \ker \overline \gamma_i - [b_{i,k}]_+ \dim \ker \overline \beta_k.
$$
Interchanging $\M$ and $\overline \M$ if necessary, we see that
it suffices to prove the following:
\begin{equation}
\label{eq:ker-gamma-bargamma}
\text{if $b_{k,i} \geq 0$ then $\dim \ker \overline \gamma_i =
\dim \ker \gamma_i - b_{k,i} \dim \ker \beta_k$.}
\end{equation}

We first show that \eqref{eq:ker-gamma-bargamma} holds if we
replace the map $\overline \gamma_i: \overline M_{\rm out}(i) \to
\overline M_{\rm in}(i)$ with its counterpart $\widetilde \gamma_i: \widetilde M_{\rm out}(i) \to
\widetilde M_{\rm in}(i)$ for the representation $\widetilde \M = \widetilde \mu_k(\M)$.
We decompose the space $M_{\rm out}(i)$ as
$$M_{\rm out}(i) = M(k)^{b_{k,i}} \oplus M'_{\rm out}(i),$$
where the first summand corresponds to the $b_{k,i}$ arrows
from~$i$ to~$k$.
Accordingly, we have
$$\widetilde M_{\rm out}(i) = M_{\rm out}(k)^{b_{k,i}} \oplus M'_{\rm out}(i), \quad
\widetilde M_{\rm in}(i) = \overline M(k)^{b_{k,i}} \oplus M_{\rm in}(i).$$
Tracing the definitions, we see that the maps
$\gamma_i: M(k)^{b_{k,i}} \oplus M'_{\rm out}(i) \to M_{\rm in}(i)$
and $\widetilde \gamma_i: M_{\rm out}(k)^{b_{k,i}} \oplus M'_{\rm out}(i) \to
\overline M(k)^{b_{k,i}} \oplus M_{\rm in}(i)$ can be written in
the block-matrix form as
$$\gamma_i = \begin{pmatrix}\psi \circ \beta_k^{b_{k,i}} & \eta\end{pmatrix},
\quad \widetilde \gamma_i =
\begin{pmatrix}\overline \alpha_k^{b_{k,i}} & 0\\
\psi & \eta\end{pmatrix}$$
for some linear maps $\psi$ and $\eta$,
where $\beta_k^{b_{k,i}}$ and $\overline \alpha_k^{b_{k,i}}$ stand for the
direct (diagonal) sums of $b_{k,i}$ copies of the maps
$\beta_k: M(k) \to  M_{\rm out}(k)$ and
$\overline \alpha_k: M_{\rm out}(k) \to \overline M(k)$.
Using the equality $\ker \overline \alpha_k = {\rm im} \ \beta_k$
(\eqref{eq:overline-kernels-images}), it is easy to see that there is
an exact sequence
$$0 \to (\ker \beta_k)^{b_{k,i}} \oplus \{0\} \to \ker \gamma_i
\to \ker \widetilde \gamma_i \to 0,$$
where the map $\ker \gamma_i \to \ker \widetilde \gamma_i$ sends a
pair $(u,v) \in \ker \gamma_i \subseteq M(k)^{b_{k,i}} \oplus M'_{\rm out}(i)$
to $(\beta_k^{b_{k,i}} u, v)$.
We conclude that
$$\dim \ker \widetilde \gamma_i =
\dim \ker \gamma_i - b_{k,i} \dim \ker \beta_k.$$

To complete the proof of \eqref{eq:ker-gamma-bargamma}, it
remains to show that
\begin{equation}
\label{eq:ker-gamma-reduces}
\dim \ker \widetilde \gamma_i =
\dim \ker \overline \gamma_i.
\end{equation}
In view of \eqref{eq:gamma-gamma'}, $\dim \ker \widetilde \gamma_i$
does not change if we replace $(\widetilde Q, \widetilde S)$ with
a right-equivalent QP.
Thus, in proving \eqref{eq:ker-gamma-reduces}, we can assume that
$(\widetilde Q, \widetilde S) = (\overline Q, \overline S)
\oplus (Q',S')$, where $(Q',S')$ is a trivial QP.
In accordance with this decomposition, we can decompose the spaces
$\widetilde M_{\rm in}(i)$ and $\widetilde M_{\rm out}(i)$ as
$$\widetilde M_{\rm in}(i) = \overline M_{\rm in}(i) \oplus \widetilde M'_{\rm in}(i),
\quad \widetilde M_{\rm out}(i) = \overline M_{\rm out}(i) \oplus \widetilde M'_{\rm out}(i),$$
where the spaces $\widetilde M'_{\rm in}(i)$ and $\widetilde M'_{\rm out}(i)$
correspond to the arrows from~$Q'$.
Thus, $\widetilde \gamma_i$ has the following block-matrix form:
$$\widetilde \gamma_i = \begin{pmatrix}\overline \gamma_i & 0\\
0 & \iota\end{pmatrix},$$
where $\iota$ is a vector space isomorphism
$\widetilde M'_{\rm out}(i) \to \widetilde M'_{\rm in}(i)$.
This implies \eqref{eq:ker-gamma-reduces}, which
completes the proofs of Lemma~\ref{lem:g-F-mutation} and Theorem~\ref{th:g-F-rep}.
\end{proof}

Theorem~\ref{th:g-F-rep} yields a formula for cluster variables in
the coefficient-free cluster algebra (that is, the one with the
coefficient semifield $\PP = \{1\}$).

\begin{corollary}
\label{cor:CC-formula}
Suppose that $F_{\ell;t}^{B;t_0} \neq 1$, hence the
QP-representation $\M = \M_{\ell;t}^{B;t_0}$ is positive
(that is, $\M = (M,0)$).
Let $x_{\ell;t}$ be the corresponding cluster variable in
the coefficient-free cluster algebra .
Then $x_{\ell;t}$ is given by the formula
\begin{equation}
\label{eq:CC-formula-general}
x_{\ell;t} = \prod_{i=1}^n x_i^{- d_i}
\sum_\e \chi(\Gr_\e (M)) \prod_{i=1}^n
x_i^{- {\rm rk} \gamma_i + \sum_j  ([b_{i,j}]_+ e_j +
[-b_{i,j}]_+ (d_j - e_j))} \ ,
\end{equation}
where $d_i = \dim M(i)$.
\end{corollary}

\begin{proof}
It suffices to rewrite \eqref{eq:g-M} as
$$g_i = \dim M_{\rm out}(i) - {\rm rk} \gamma_i - \dim M(i)
= \sum_j [-b_{i,j}]_+ d_j - {\rm rk} \gamma_i - d_i  \ ,$$
and apply \eqref{eq:xjt=F/F} and \eqref{eq:yhat-0}.
\end{proof}

\begin{remark}
If the quiver $Q(B)$ has no oriented cycles then $S = 0$, hence
$\gamma_i = 0$ for all $i$.
In this case \eqref{eq:CC-formula-general} specializes to the
Caldero-Chapoton formula for cluster variables
(see~\cite{cc})
 obtained in this
generality in \cite[Theorem~3]{ck}.
\end{remark}

Recall that the \emph{denominator vector} of a cluster
variable~$z$ with respect to the initial cluster $(x_1, \dots, x_n)$
is the integer vector $(d_1(z), \dots, d_n(z))$ such that
$$z = \frac{P(x_1, \dots, x_n)}{x_1^{d_1(z)} \cdots x_n^{d_n(z)}},$$
where $P$ is a polynomial not divisible by any~$x_i$.
Conjecture~7.17 in \cite{ca4} claims that if $z$ does not belong to the initial cluster
then the denominator vector of~$z$ is equal to the multidegree of the corresponding
$F$-polynomial.
By Proposition~\ref{pr:F-ct-ht} and Theorem~\ref{th:g-F-rep}, this
conjecture is equivalent to the equality
\begin{equation}
\label{eq:denom-dim}
d_i(x_{\ell;t}) = d_i = \dim M(i)
\end{equation}
(in the notation of Corollary~\ref{cor:CC-formula}).
It was shown in \cite{ck} in the case where $Q(B)$ has no oriented cycles,
that $\eqref{eq:CC-formula-general}$ implies
$\eqref{eq:denom-dim}$. A direct proof of this was given in~\cite[Theorem 10]{hubery}.
In full generality, \eqref{eq:denom-dim} was disproved by a counterexample in \cite{fk}
(based on the ideas in \cite{bmr}).
Using Theorem~\ref{th:g-F-rep}, we obtain the following partial result.

\begin{corollary}
\label{cor:denom-dim-inequality}
In the notation of Corollary~\ref{cor:CC-formula}, we have the inequality
\begin{equation}
\label{eq:denom-dim-inequality}
d_i(x_{\ell;t}) \leq d_i \ .
\end{equation}
Furthermore, a necessary condition for the equality in
\eqref{eq:denom-dim-inequality} is the existence of a quiver
subrepresentation $N$ of $M$ such that
\begin{equation}
\label{eq:necessary-den-dim-equality}
\ker \gamma_i \subseteq N_{\rm out}(i), \quad
\gamma_i ( N_{\rm out}(i)) =  N_{\rm in}(i) \ .
\end{equation}
\end{corollary}

\begin{proof}
In view of \eqref{eq:CC-formula-general}, we have
\begin{equation}
\label{eq:denom-dim-difference}
d_i - d_i(x_{\ell;t}) =
\min_{\e} (- {\rm rk} \gamma_i + \sum_j  ([b_{i,j}]_+ e_j +
[-b_{i,j}]_+ (d_j - e_j))) \ ,
\end{equation}
where the minimum is over all dimension vectors $\e$ such that
$\chi(\Gr_\e (M)) \neq 0$.
In particular, $\Gr_\e (M)$ must be nonempty, i.e., $M$
must have a subrepresentation $N$ with $e_i = \dim N(i)$ for all $i$.
In terms of $N$, we have
$$\sum_j  [b_{i,j}]_+ e_j = \dim N_{\rm in}(i), \quad
\sum_j  [-b_{i,j}]_+ (d_j - e_j) = \dim M_{\rm out}(i) -
\dim N_{\rm out}(i) \ .$$
Therefore,
\begin{align*}
&- {\rm rk} \gamma_i + \sum_j  ([b_{i,j}]_+ e_j +
[-b_{i,j}]_+ (d_j - e_j))\\
&= - {\rm rk} \gamma_i
+ \dim M_{\rm out}(i) + \dim N_{\rm in}(i) -
\dim N_{\rm out}(i)\\
&= \dim \ker \gamma_i + \dim N_{\rm in}(i) -
\dim N_{\rm out}(i)\\
&= \dim \frac{\ker \gamma_i}{\ker \gamma_i \cap N_{\rm out}(i)}
+ \dim \frac{N_{\rm in}(i)}{\gamma_i (N_{\rm out}(i))} \ ,
\end{align*}
making clear both assertions in question.
\end{proof}

\begin{remark}
A counterpart of Corollary~\ref{cor:denom-dim-inequality} in the
context of $2$-Calabi-Yau categories was obtained in
\cite[Proposition~5.8]{fk}.
\end{remark}

We conclude this section by applying the above results for an explicit construction
of a special class of QP-representations corresponding to cluster variables.
Let $T$ be a subset of vertices of $Q = Q(B)$ such that the induced
subgraph on~$T$ is a tree; in particular, $b_{i,j} \in \{0, \pm
1\}$ for $i,j \in T$, so inside~$T$ there are no multiple arrows.
Without loss of generality, we can assume that $T = [1,\ell]
\subseteq [1,n] = Q_0$, and that each $i \in T$ is a
\emph{leaf} of the subtree of $T$ on vertices $[i,\ell]$; in other
words, for each $i \in [1,\ell-1]$ there is a a unique $j \in [i+1,\ell]$
connected by an edge with~$i$.
Let $M = M_T$ be a $Q$-representation such that $M(i) = \C$ for $i
\in T$, $M(i) = 0$ for $i \notin T$, and $a_M: M(t(a)) \to
M(h(a))$ is an isomorphism whenever $h(a)$ and $t(a)$ belong to~$T$.
The condition that~$T$ is a tree implies that $M$ is a
$\Pcal(Q,S)$-module for any potential~$S$ (since every cyclic
derivative $\partial_a S$ is a linear combination of paths from
$h(a)$ to $t(a)$, and every such path acts as~$0$ in~$M$).

\begin{proposition}
\label{pr:one-dim-reps}
Let $t_0 \overunder{1}{} t_1 \overunder{2}{}  \cdots \overunder{\ell}{} t_\ell = t$
be a path in $\TT_n$.
Then
\begin{equation}
\label{eq:dim-one-rep-cluster-variable}
\M_{\ell;t}^{B;t_0} = (M_T, 0) \ .
\end{equation}
\end{proposition}

\begin{proof}
We need to show that the sequence of mutations $\mu_\ell \circ
\cdots\circ\mu_1$ takes the QP-re\-pre\-sen\-ta\-tion $(M_T, 0)$ to the
negative simple representation $\mathcal S_\ell^-$ (see
\eqref{eq:neg-simple}).
For $\ell = 1$, the representation $M_T$ is just the (positive)
simple module $\mathcal S_\ell$; using \eqref{eq:new-Mk}, we see that
the mutation $\mu_\ell$ turns it into $\mathcal S_\ell^-$.
For $\ell > 1$, again using \eqref{eq:new-Mk}, we see that the mutation $\mu_1$
turns $M_T$ into $M_{T'}$, where the tree $T'$ is obtained from~$T$
by removing the leaf~$1$.
The proof is finished by induction on~$\ell$.
\end{proof}

\begin{corollary}
\label{cor:F-g-d-one-dim}
In the situation of Proposition~\ref{pr:one-dim-reps}, the
$F$-polynomial $F_{\ell;t}^{B;t_0}$ is given as
follows:
\begin{equation}
\label{eq:F-dim-one}
F_{\ell;t}^{B;t_0}(u_1, \dots, u_n) = \sum_Z \prod_{i \in Z}u_i \ ,
\end{equation}
where $Z$ runs over all subsets of~$T = [1, \ell]$ with the
property that if $j \in Z$ then $i \in Z$ for every arrow $j \to i$ in~$T$.
Furthermore, the denominator vector of the cluster variable $x_{\ell;t}$
is the indicator vector of $[1,\ell]$ (that is, $d_i(x_{\ell;t}) = 1$
for $i \in [1,\ell]$, and $d_i(x_{\ell;t}) = 0$ for $i \in [\ell+1,n]$).
\end{corollary}

\begin{proof}
By Theorem~\ref{th:g-F-rep} and Proposition~\ref{pr:one-dim-reps},
we have $F_{\ell;t}^{B;t_0} = F_{M_T}$.
The equality \eqref{eq:F-dim-one} is then immediate from the
definition \eqref{eq:F-M}: clearly, the quiver Grassmannian
$\Gr_\e (M_T)$ consists of one point if $\e$ is the indicator
vector of a subset $Z$ as in \eqref{eq:F-dim-one}, otherwise
$\Gr_\e (M_T) = \emptyset$.

Turning to the denominator vector, in view of
Corollary~\ref{cor:denom-dim-inequality}, it is enough to show
that $d_i(x_{\ell;t}) = 1$ for $i \in T$. Fix a vertex $i \in T$,
and let $Z$ be the subset of all vertices $j \in T$ that can be
reached from~$i$ by a directed path in~$T$. Let $N = \oplus_{j \in
Z} M_T(j)$. Then $N$ is a quiver subrepresentation of $M_T$. The
fact that $T$ is a tree implies easily that $N$ satisfies
\eqref{eq:necessary-den-dim-equality} (indeed, we have $\gamma_i =
0$,
 $N_{\rm out}(i) =  M_{\rm out}(i)$, and $N_{\rm in}(i) = 0$).
Furthermore, $N$ is the only element in its quiver Grassmannian,
which makes \eqref{eq:necessary-den-dim-equality} not only
necessary but also a sufficient condition for the equality
$d_i(x_{\ell;t}) = d_i(M) = 1$.
\end{proof}

\begin{remark}
\label{rem:g-dim-one}
The computation of the $\g$-vector of $M_T$ is more involved,
since the map $\gamma_i$ is not necessarily~$0$ if $i \notin T$.
However, $\gamma_i = 0$ for $i \in T$, hence for $i \in T$ the component
$g_i$ of $\g_{\ell;t}^{B;t_0}$ is equal to $|\{j \in T: i \to j\}| - 1$.
\end{remark}

\section{Mutations preserve homomorphisms modulo confined ones}
\label{sec:homs-modulo-confined}

Let $\M = (M,V)$ and $\N = (N,W)$ be QP-representations of a reduced QP $(Q,S)$.
We fix a vertex $k \in Q_0$ and assume that~$Q$ has no oriented
$2$-cycles through~$k$.
Thus, the mutated QP $(\overline Q, \overline S) = \mu_k(Q,S)$ is
well-defined, as well as its QP-representations
$\overline \M = (\overline M,\overline V) = \mu_k(\M)$ and
$\overline {\N} = (\overline {N},\overline {W}) = \mu_k(\N)$.

We abbreviate
$$M(\widehat k) = \bigoplus_{i \neq k} M(i),$$
and say that a homomorphism $\varphi \in \Hom_Q (M,N)$
is \emph{confined to $k$} if $\varphi (m) = 0$ for
$m \in M(\widehat k)$.
Denote the space of such homomorphisms
by $\Hom_Q^{[k]}(M,N)$.
Restricting $\varphi$ to $M(k)$ yields a vector space isomorphism
\begin{equation}
\label{eq:confined-homs}
\Hom_Q^{[k]}(M,N) = \Hom_{\C}({\rm coker} \,\alpha_{k;M}, \ker \beta_{k;N}).
\end{equation}

The goal of this section is to prove the following proposition.
It was already established in \cite[Theorem~7.1]{birs} but the
present proof seems to be much simpler.

\begin{proposition}
\label{pr:mutation-preserves-homs-mod-k}
The mutation~$\mu_k$ induces an isomorphism
$$\Hom_Q(M,N)/\Hom_Q^{[k]}(M,N) =
\Hom_{\overline Q}(\overline M, \overline N)/
\Hom_{\overline Q}^{[k]}(\overline M,\overline N).$$
\end{proposition}

\begin{proof}
We can view a ${\mathcal P}(Q,S)$-module~$M$ as a module over the subalgebra
$${\mathcal P}(Q,S)_{\widehat k, \widehat k} =
\bigoplus_{i,j \neq k} {\mathcal P}(Q,S)_{i, j}.$$
Clearly, $M(\widehat k)$ is a
${\mathcal P}(Q,S)_{\widehat k, \widehat k}$-submodule of~$M$, and
so we have the restriction map
$\rho: \Hom_Q(M,N) \to \Hom_{{\mathcal
P}(Q,S)_{\hat k, \hat k}}(M(\widehat k),N(\widehat k))$.
Denote
\begin{align*}
\Hom_{Q(\widehat k)}(M,N) =&
\{\varphi \in
\Hom_{{\mathcal P}(Q,S)_{\hat k, \hat k}}(M(\widehat k),N(\widehat k)):\\
&\varphi(\ker \alpha_{k;M}) \subseteq \ker \alpha_{k;N}, \,\,
\varphi({\rm im}\, \beta_{k;M}) \subseteq  {\rm im}\, \beta_{k;N}\}.
\end{align*}
As an easy consequence of the definitions, we have
$$\ker \rho = \Hom_Q^{[k]}(M,N), \quad
{\rm im} \, \rho = \Hom_{Q(\widehat k)}(M,N).$$
Thus, $\rho$ induces an isomorphism
\begin{equation}
\label{eq:modding-k-out}
\Hom_Q(M,N)/\Hom_Q^{[k]}(M,N) = \Hom_{Q(\widehat k)}(M,N).
\end{equation}

Now recall from \cite[Proposition~6.1, Corollary~6.6]{dwz}
that the mutation~$\mu_k$ induces an isomorphism between
${\mathcal P}(Q,S)_{\widehat k, \widehat k}$ and
${\mathcal P}(\overline Q, \overline S)_{\widehat k, \widehat k}$.
This isomorphism is explicitly described in the proof of Proposition~6.1
in \cite{dwz}: it preserves all arrows not incident to~$k$, and
sends each product $ba$ (for~$a$ an incoming, and $b$ an outgoing
arrow at~$k$) to the ``composite arrow" $[ba]$.
Identifying ${\mathcal P}(Q,S)_{\widehat k, \widehat k}$ and
${\mathcal P}(\overline Q, \overline S)_{\widehat k, \widehat k}$
with the help of this isomorphism, and
recalling the definition of $\overline M$ in
Section~\ref{sec:QP-background}, we see that the
${\mathcal P}(\overline Q, \overline S)_{\widehat k, \widehat
k}$-module structure on $\overline M(\widehat k)$ becomes identical to
the ${\mathcal P}(Q,S)_{\widehat k, \widehat k}$-module structure
on $M(\widehat k)$.
Furthermore, by \eqref{eq:overline-kernels-images} we have
$\ker \overline \alpha_k = {\rm im} \, \beta_k, \quad
{\rm im} \,  \overline \beta_k = \ker \alpha_k$.
Therefore, the subspace
$ \Hom_{Q(\widehat k)}(M,N) \subseteq
\Hom_{{\mathcal P}(Q,S)_{\hat k, \hat k}}(M(\widehat k),N(\widehat k))$
gets identified with $ \Hom_{\overline Q(\widehat k)}(\overline M, \overline {N})$.
This completes the proof of Proposition~\ref{pr:mutation-preserves-homs-mod-k}.
\end{proof}

The isomorphism in Proposition~\ref{pr:mutation-preserves-homs-mod-k}
can be viewed as functorial in the following way.
Let ${\mathcal C}(Q,S)$ be the category whose objects are
QP-representations
of a QP $(Q,S)$, and the
morphisms are given by
\begin{equation*}
\Hom_{{\mathcal C}(Q,S)}((M,V),(N,W)) = \Hom_Q (M,N)\oplus \Hom_R (V,W).
\end{equation*}
For a vertex $k \in Q_0$, let ${\mathcal C}^{[\widehat k]}(Q,S)$
be the quotient category of ${\mathcal C}(Q,S)$ with the same
objects, and the morphisms given by
\begin{equation*}
\Hom_{{\mathcal C}^{[\hat k]}(Q,S)}((M,V),(N,W)) =
\frac{\Hom_{{\mathcal C}(Q,S)}((M,V),(N,W))}{\Hom_Q^{[k]} (M,N)\oplus \Hom_R (V,W)}\,.
\end{equation*}

\begin{proposition}
\label{pr:mutation-equivalence-of-categories}
The mutation $\mu_k$ induces an equivalence of categories
$$
\mu_k: {\mathcal C}^{[\widehat k]}(Q,S) \to
{\mathcal C}^{[\widehat k]}(\overline Q, \overline S).
$$
\end{proposition}

\begin{proof}
In view of \eqref{eq:modding-k-out}, we have
$$\Hom_{{\mathcal C}^{[\hat k]}(Q,S)}((M,V),(N,W)) = \Hom_{Q(\widehat k)}(M,N).$$
It follows from the proof of Proposition~\ref{pr:mutation-preserves-homs-mod-k}
that the mutation at~$k$ gives rise to a functor from ${\mathcal C}^{[\widehat k]}(Q,S)$
to ${\mathcal C}^{[\widehat k]}(\overline Q, \overline S)$.
The fact that this functor is an equivalence of categories
is a consequence of the following  basic result in the category theory (see~\cite[Proposition 16.3.2]{schubert}).

\begin{proposition}
\label{pr:Equivalence}
Let ${\mathcal C}$ and $\overline {\mathcal C}$ be categories, and
suppose ${\mathcal F}:{\mathcal C}\to \overline {\mathcal C}$ is a functor with the following properties:
\begin{enumerate}
\item For every object ${\overline{\M}}$ of $\overline {\mathcal C}$ there is an object
${\M}$ of ${\mathcal C}$
such that ${\mathcal F}({\M})$ is isomorphic to $\overline{\mathcal M}$;
\item For any pair of objects ${\mathcal M},{\mathcal N}$ of ${\mathcal C}$, the functor
${\mathcal F}$ induces a bijection
$$
\Hom_{{\mathcal C}}({\mathcal M},{\mathcal N})\cong
\Hom_{\overline{\mathcal C}}({\mathcal F}{\mathcal M},{\mathcal F}{\mathcal N}).
$$
\end{enumerate}
Then ${\mathcal F}$ is an equivalence of categories, i.e., there exists a functor
${\mathcal G}:\overline{\mathcal C}\to {\mathcal C}$
such that the the composition functors ${\mathcal G}\circ {\mathcal F}$ and
${\mathcal F}\circ {\mathcal G}$ are naturally equivalent
to the identity functors of ${\mathcal C}$ and $\overline{\mathcal C}$, respectively.
\end{proposition}
\end{proof}

\begin{remark}
\label{rem:axiom_of_choice}
The proof of Proposition~\ref{pr:Equivalence} is based on a strong version of the axiom of choice
(see~\cite[\S3.1, Remark 16.3.3]{schubert}):
for any class of sets and any equivalence relation on this class we can choose a representative in every class.
\end{remark}

\section{The $E$-invariant}
\label{sec:E}

Let ${\mathcal M}=(M,V)$ and ${\mathcal N}=(N,W)$ be
QP-representations of the same nondegenerate QP $(Q,S)$.
We abbreviate
\begin{equation}
\label{eq:abbreviations}
\langle M, N \rangle = \dim \Hom_Q(M,N),
\end{equation}
and
\begin{equation}
\label{eq:abbreviations-2}
d_i(\M) = d_i(M) = \dim M(i), \quad d_i^-(\M) = \dim V(i),
\end{equation}
so that the components of the $\g$-vector $\g_\M = (g_1, \dots, g_n)$ are given by
\begin{equation}
\label{eq:gi-abbreviated}
g_i = g_i (\M) = \dim \ker \gamma_{i;M} - d_i(\M) + d_i^-(\M).
\end{equation}
We now define the integer function
\begin{equation}
\label{eq:E-abbreviated}
E^\inj(\M, \N) = \langle M, N \rangle + \sum_{i=1}^n d_i(\M) g_i(\N),
\end{equation}
and its symmetrized version
\begin{equation}
\label{eq:E-sym}
E^{\rm sym}(\M, \N) = E^\inj(\M, \N) + E^\inj(\N, \M).
\end{equation}

In view of \eqref{eq:Einvariant}, the $E$-invariant of a QP-representation
is given by
\begin{equation}
\label{eq:E-M-MN}
E(\M )=E^\inj(\M, \M) = \frac{E^{\rm sym}(\M, \M)}{2}.
\end{equation}

Now let $\mu_k(\M) = \overline \M = (\overline M, \overline V)$
and $\mu_k(\N) = \overline \N = (\overline N, \overline W)$ be
QP-representations (of the QP $(\overline Q, \overline S) =
\mu_k(Q,S)$) obtained from $\M$ and $\N$ by the mutation at a
vertex~$k$.

\begin{theorem}
\label{th:E-invariant}
We have
\begin{equation}
\label{eq:EMN-numeric}
E^\inj(\overline{\M},\overline{\N})-E^\inj(\M,\N)= h_k(\overline M)h_k(N)-h_k(M)h_k(\overline N).
\end{equation}
In particular, $E^{\rm sym}(\M, \N)$ and $E({\mathcal M})$ are invariant under QP-mutations, i.e.,
$$
E^{\rm sym}(\mu_k(\M), \mu_k(\N)) = E^{\rm sym}(\M, \N), \quad
E(\mu_k({\mathcal M}))=E({\mathcal M})
$$
for any vertex~$k$.
\end{theorem}

\begin{proof}
Our starting point is the equality
\begin{equation}
\label{eq:mod-simples-at-k-invariant}
\langle M,N \rangle + \dim ({\rm coker}\, \alpha_{k;M}) \cdot
h_k(N) = \langle \overline M, \overline N \rangle +
\dim ({\rm coker}\, \alpha_{k;\overline M})
\cdot h_K(\overline N) \ ,
\end{equation}
obtained by combining Proposition~\ref{pr:mutation-preserves-homs-mod-k} with
\eqref{eq:confined-homs} (and recalling the notation
\eqref{eq:h-ker-beta}).
We claim that $\dim ({\rm coker}\, \alpha_{k;M})$
and $\dim ({\rm coker}\, \alpha_{k;\overline M})$ are given by
\begin{equation}
\label{eq:dim-coker-alpha}
\dim ({\rm coker}\, \alpha_{k;M}) = h_k(\overline M) +
d_k(M) + d_k(\overline M) - \sum_{i} [-b_{i,k}]_+ d_i (M)  \ ,
\end{equation}
and
\begin{equation}
\label{eq:dim-coker-overline-alpha}
\dim ({\rm coker}\, \alpha_{k;\overline M}) =
h_k(M) +
d_k(M) + d_k(\overline M) - \sum_{i} [b_{i,k}]_+ d_i (M) \ .
\end{equation}
Note that \eqref{eq:dim-coker-overline-alpha} follows from
\eqref{eq:dim-coker-alpha} by interchanging $M$ with $\overline M$,
so it is enough to prove \eqref{eq:dim-coker-alpha}.
Using the equality $\ker \alpha_{k;M} = {\rm im}\, \beta_{k;\overline M}$
in \eqref{eq:overline-kernels-images}, we obtain
\begin{align*}
\dim ({\rm coker}\, \alpha_{k;M}) &= d_k(M) - {\rm rk}\, \alpha_{k;M}\\
&= d_k(M) - \dim M_{\rm in}(k) + \dim (\ker \alpha_{k;M})\\
&= d_k(M) - \sum_{i} [-b_{i,k}]_+ d_i (M) + {\rm rk}\, \beta_{k;\overline M}\\
&= d_k(M) - \sum_{i} [-b_{i,k}]_+ d_i (M) + d_k(\overline M) -
\dim (\ker \beta_{k;\overline M})\ ,
\end{align*}
which implies \eqref{eq:dim-coker-alpha} in view of
\eqref{eq:h-ker-beta}.

Using \eqref{eq:dim-coker-alpha} and
\eqref{eq:dim-coker-overline-alpha}, we can rewrite
\eqref{eq:mod-simples-at-k-invariant} as follows:
\begin{align*}
\langle \overline M, \overline N \rangle - \langle M,N \rangle
&= \dim ({\rm coker}\, \alpha_{k;M}) \cdot h_k(N) -
\dim ({\rm coker}\, \alpha_{k;\overline M}) \cdot h_K(\overline N)
\\&= (h_k(\overline M) +
d_k(M) + d_k(\overline M) - \sum_{i} [-b_{i,k}]_+ d_i (M)) \cdot h_K(N)
\\& - (h_k(M) + d_k(M) + d_k(\overline M) - \sum_{i} [b_{i,k}]_+ d_i(M))
\cdot h_K(\overline N) \ .
\end{align*}
In view of Lemma~\ref{lem:g-F-mutation}, we have $h_k(\overline N)
= h_k(N) - g_k(\N)$, which allows us to rewrite \eqref{eq:mod-simples-at-k-invariant}
further as
\begin{align}
\label{eq:prelim-identity}
&\langle \overline M, \overline N \rangle - \langle M,N \rangle
- (h_k(\overline M)h_k(N)-h_k(M)h_k(\overline N))\\
\nonumber
&=
(\sum_{i} b_{i,k} d_i(M)) \cdot h_k(N) +
(d_k(M) + d_k(\overline M) - \sum_{i} [b_{i,k}]_+ d_i(M)) \cdot g_k(\N)\ .
\end{align}
Comparing \eqref{eq:prelim-identity} with the desired
equality \eqref{eq:EMN-numeric}, we see that it remains to
show that the right hand side of \eqref{eq:prelim-identity} is equal to
$$\sum_{i} (d_i(M) g_i(\N) - d_i(\overline M) g_i(\overline \N)).$$
Using the equality $d_i(\overline M) = d_i(M)$
for $i \neq k$, and the assertion (proved in Lemma~\ref{lem:g-F-mutation})
that the transformation $\g_{\N} \mapsto \g_{\overline \N}$ is given
by \eqref{eq:g-transition}, we obtain
\begin{align*}
\sum_{i}& (d_i(M) g_i(\N) - d_i(\overline M) g_i(\overline \N))\\
&= (d_k(M) + d_k(\overline M)) \cdot g_k(\N) +
\sum_{i \neq k} d_i(M) (b_{i,k} h_k(N) - [b_{i,k}]_+ g_k(\N))\\
&=
(\sum_{i} b_{i,k} d_i(M)) \cdot h_k(N) + (d_k(M) + d_k(\overline M) - \sum_{i} [b_{i,k}]_+
d_i(M)) \cdot g_k(\N) \ ,
\end{align*}
finishing the proof of Theorem~\ref{th:E-invariant}.
\end{proof}

\begin{corollary}
\label{cor:EMzero}
If ${\mathcal M}$ is obtained by a sequence of mutations
from a negative QP-re\-pre\-sen\-ta\-tion $(\{0\},V)$, then $E({\mathcal M})=0$.
In particular, this is the case for any representation $\M_{\ell;t}^{B;t_0}$
given by \eqref{eq:M-cluster-variable}.
\end{corollary}

\begin{proof}
By the definition \eqref{eq:Einvariant}, we have $E((\{0\},V))=0$, hence
$E({\mathcal M})=0$ as well.
\end{proof}

We conclude this section by one more invariance property of $E(\M)$.
For a quiver $Q=(Q_0,Q_1,h,t)$, we denote by $Q^{\rm op}$ the
\emph{opposite} quiver $(Q_0,Q_1,t,h)$ obtained from $Q$ by
reversing all arrows.
To distinguish the arrows of $Q^{\rm op}$ from those of $Q$, we
denote by $a^{\rm op}$ the arrow of $Q^{\rm op}$ corresponding to
an arrow~$a$ of~$Q$.
The correspondence $a \mapsto a^{\rm op}$ extends to an
anti-isomorphism $u \mapsto u^{\rm op}$ of completed path algebras
$R\langle \langle A \rangle \rangle \to
R\langle \langle A^{\rm op} \rangle \rangle$ (identical on the vertex
span~$R$).
In particular, every QP $(Q,S)$ gives rise to the
opposite QP $(Q^{\rm op}, S^{\rm op})$.
By the definition \eqref{eq:cyclic-derivative} of a cyclic
derivative, we have
$\partial_{a^{\rm op}} S^{\rm op} = (\partial_{a} S)^{\rm op}$
for any arrow~$a$ of~$Q$.
Thus, $J(S^{\rm op}) = (J(S))^{\rm op}$.
This implies that every QP-representation $\M = (M,V)$ of $(Q,S)$
gives rise to a QP-representation $\M^\star = (M^\star, V)$
of $(Q^{\rm op}, S^{\rm op})$ obtained from~$\M$ by replacing each
space $M(k)$ with its dual $M(k)^\star$, and setting
$(a^{\rm op})_{M^\star} = (a_M)^\star$ for any arrow $a$ of $Q$.

\begin{proposition}
\label{pr:E-invariant-star}
We have $E(\M^\star) = E(\M)$ for any QP-representation $\M$.
\end{proposition}

\begin{proof}
Using the notation in \eqref{eq:abbreviations-2} and
\eqref{eq:gi-abbreviated}, we can express $g_i(\M)$ as
\begin{equation}
\label{eq:gi-thru-rank-gamma}
g_i(\M) = \dim M_{\rm out}(i) - {\rm rk}\ \gamma_{i;M} - d_i(M) + d_i^-(\M),
\end{equation}
and $E(\M)$ as
\begin{align}
\label{eq:E-thru-rank-gamma}
E(\M) &= \langle M, M \rangle + \sum_{i=1}^n d_i(M) (d_i^-(\M) -
{\rm rk}\ \gamma_{i;M} - d_i(M))\\
\nonumber
& + \sum_{a \in Q_1} d_{h(a)}(M) d_{t(a)}(M) \ .
\end{align}
It remains to observe that passing from $\M$ to $\M^\star$
does not change any of the terms in \eqref{eq:E-thru-rank-gamma}
(since $\Hom_{Q^{\rm op}} (M^\star, M^\star)$ is isomorphic to
$\Hom_{Q} (M, M)$, and $\gamma_{i;M^\star}= (\gamma_{i;M})^\star$).
\end{proof}

\section{Lower bounds for the $E$-invariant}
\label{sec:inequalities}
Fix a QP-representation $\M = (M,V)$ of a reduced QP $(Q,S)$.
The goal of this section is to prove the lower bound \eqref{eq:E-lower-bound-beta}
for $E(\M)$.
Using the notation in \eqref{eq:abbreviations-2}, we can state the
result as follows (since $\M$ is fixed, we
allow ourselves to skip references to it in most of the formulas below).

\begin{theorem}
\label{th:E-lower-bound-beta}
The $E$-invariant of a QP-representation satisfies
\begin{equation}
\label{eq:E-lower-bound-beta-abbreviated}
E(\M) \geq \sum_{i \in Q_0}(\dim (\ker \beta_i) \cdot \dim (\ker\gamma_i/\im \beta_i)+
 d_i(M) \cdot d_i^-(\M)).
\end{equation}
\end{theorem}

\begin{proof}
The desired lower bound for $E(\M)$ follows from another one:
\begin{equation}
\label{eq:E-lower-bound-alpha}
E(\M) \geq \sum_{i \in Q_0}(\dim ({\rm coker}\ \alpha_i) \cdot \dim (\ker\alpha_i/\im \gamma_i)
+  d_i(M) \cdot d_i^-(\M));
\end{equation}
indeed, to deduce \eqref{eq:E-lower-bound-beta-abbreviated} from \eqref{eq:E-lower-bound-alpha}
it suffices to apply the latter bound to the dual QP-representation $\M^\star$ and use
Proposition~\ref{pr:E-invariant-star}.

Substituting into \eqref{eq:E-lower-bound-alpha} the expression \eqref{eq:E-thru-rank-gamma}
for $E(\M)$, regrouping the terms and simplifying, we can rewrite it as follows:
\begin{align}
\label{eq:E-lower-bound-alpha-rewritten}
& \langle M, M \rangle +
 \sum_{a \in Q_1} d_{h(a)}(M)\cdot d_{t(a)}(M)\\
 \nonumber &   -
\sum_{i \in Q_0}(d_i(M)^2 + \dim ({\rm coker}\ \alpha_i) \cdot \dim (\ker\alpha_i))
\geq \sum_{i \in Q_0} {\rm rk} \gamma_i \cdot {\rm rk} \alpha_i.
\end{align}

We abbreviate
$$U = \bigoplus_{i \in Q_0} \Hom_\C (M_{\rm in}(i), M(i)),$$
and define the subspaces $U_1 \subseteq U_2$ in~$U$ by
\begin{align}
\label{eq:U1}
U_1 & = \{(\psi_i: M_{\rm in}(i) \to M(i))_{i \in Q_0} : \
{\rm im} \ \psi_i \subseteq {\rm im} \ \alpha_i \,\,\text{for all}\,\, i\},\\
\nonumber
U_2 & = \{(\psi_i: M_{\rm in}(i) \to M(i))_{i \in Q_0} : \
\psi_i (\ker \alpha_i) \subseteq {\rm im} \ \alpha_i \,\,\text{for all}\,\, i\}.
\end{align}
Now we can state the key lemma.

\begin{lemma}
\label{lem:two-maps}
There exist two linear maps
\begin{equation*}
\xymatrix{
\Hom_R (M, M) \ar[r]^-\Phi & U  \ar[r]^-\Psi
& \bigoplus_{b\in Q_1}
\Hom_\C (M(h(b)), M(t(b)))}
\end{equation*}
satisfying the following conditions:
\begin{enumerate}
\item $\ker \Phi = \Hom_Q (M,M)$;
\item ${\rm \im} \ \Phi \subseteq U_2$;
\item ${\rm \im} \ \Phi \subseteq \ker \Psi$;
\item $\dim \Psi (U_1) \geq \sum_{i \in Q_0} {\rm rk}\ \gamma_i \cdot {\rm rk}\ \alpha_i$.
\end{enumerate}
\end{lemma}

Before proving Lemma~\ref{lem:two-maps}, we show that it implies
\eqref{eq:E-lower-bound-alpha-rewritten}.
By the definition of $U_2$, we have
\begin{align*}
\dim U_2 &= \dim U - \sum_{i \in Q_0} \dim ({\rm coker}\ \alpha_i) \cdot \dim (\ker\alpha_i)\\
& =  \sum_{a \in Q_1} d_{h(a)}(M)\cdot d_{t(a)}(M)   -
\sum_{i \in Q_0} \dim ({\rm coker}\ \alpha_i) \cdot \dim (\ker\alpha_i).
\end{align*}
Note also that
$$\Hom_R (M, M) = \bigoplus_{i\in Q_0} \End_\C (M(i)).$$
By (1), we have
$${\rm rk} \ \Phi = \dim (\Hom_R (M, M))
 - \dim (\ker \Phi) = \sum_{i \in Q_0} d_i(M)^2 - \langle M, M  \rangle.$$
In view of (2), the left hand side of \eqref{eq:E-lower-bound-alpha-rewritten}
is equal to $\dim (U_2/{\rm \im} \ \Phi)$.
Now we use (3) and (4) to conclude that
\begin{align*}
\dim (U_2/{\rm \im} \ \Phi) &\geq \dim (U_2/(U_2 \cap \ker \Psi)) = \dim (\Psi(U_2))\\
&\geq \dim (\Psi(U_1))  \geq \sum_{i \in Q_0} {\rm rk}\ \gamma_i \cdot {\rm rk}\ \alpha_i,
\end{align*}
finishing the proof of \eqref{eq:E-lower-bound-alpha-rewritten}.

\smallskip

To complete the proof of Theorem~\ref{th:E-lower-bound-beta} it
remains to prove Lemma~\ref{lem:two-maps}.
We define the map $\Phi$ by setting, for $\xi \in \Hom_R (M, M)$,
\begin{equation}
\label{eq:Phi}
\Phi(\xi) = (\eta_i: M_{\rm in}(i) \to M(i))_{i \in Q_0} \in
U,\quad \eta_i = \xi \alpha_i - \alpha_i \xi.
\end{equation}
Properties (1) and (2) from  Lemma~\ref{lem:two-maps} are
immediate from this definition.

The definition of~$\Psi$ requires some preparation.
First of all, we identify the space $U$ with
$\bigoplus_{a \in Q_1} \Hom_\C (M(t(a)), M(h(a)))$, so view $\Psi$ as
a linear map
$$\Psi: \bigoplus_{a \in Q_1} \Hom_\C (M(t(a)), M(h(a)))
 \to \bigoplus_{b\in Q_1}
\Hom_\C (M(h(b)), M(t(b)))\ .$$
Now recall from \cite[(3.2)]{dwz} that each arrow $a \in Q_1$
gives rise to a continuous linear map
$$
\Delta_a: R\langle\langle A\rangle\rangle\to
R\langle\langle A\rangle\rangle \ \widehat{\otimes}\ R\langle\langle A\rangle \rangle \ ,
$$
such that for every path $a_{1} \cdots a_{d}$ we have
\begin{equation}
\label{eq:Delta-b-cycle}
\Delta_a (a_{1} \cdots a_{d}) =\sum_{p: a_p=a} a_{1}\cdots a_{p -1} \otimes a_{p+1}\cdots a_{d} \ ;
\end{equation}
here we use the notation
$$R\langle\langle A\rangle\rangle \ \widehat{\otimes}\ R\langle\langle A\rangle\rangle=
\prod_{i,j\geq 0} (A^{\otimes_R i} \otimes A^{\otimes_R j}) \ ,
$$
and the convention that if $a_1 = a$ (resp. $a_d = a$) then the
corresponding term in \eqref{eq:Delta-b-cycle} is
$e_{h(a)} \otimes a_{2}\cdots a_{d}$ (resp. $a_{1}\cdots a_{d-1} \otimes e_{t(a)}$).
In particular, for every $a,b \in Q_1$, we have
$$\Delta_a(\partial_b S) \in R\langle\langle A\rangle\rangle_{t(b), h(a)}
\ \widehat{\otimes}\ R\langle\langle A\rangle \rangle_{t(a), h(b)} \ ;$$
accordingly, we express $\Delta_a(\partial_b S)$ as
\begin{equation}
\label{eq:delta-partial-S}
\Delta_a(\partial_b S) = \sum_\nu u_{b,a}^{(\nu)} \otimes
v_{a,b}^{(\nu)} \quad (u_{b,a}^{(\nu)} \in R\langle\langle A\rangle\rangle_{t(b),h(a)},
\,\, v_{a,b}^{(\nu)} \in R\langle\langle A\rangle \rangle_{t(a),h(b)})\ .
\end{equation}
Now we define the component
$\Psi_{b,a}: \Hom_\C (M(t(a)), M(h(a))) \to \Hom_\C (M(h(b)), M(t(b)))$
of $\Psi$ by setting
\begin{equation}
\label{eq:Psi-ba}
\Psi_{b,a}(\eta_a) = \sum_\nu (u_{b,a}^{(\nu)})_M \circ \eta_a \circ  (v_{a,b}^{(\nu)})_M \ .
\end{equation}

We postpone the proof of  property (3) in
Lemma~\ref{lem:two-maps}, that is, of the equality $\Psi \circ \Phi = 0$, until
Section~\ref{sec:homological}, see Corollary~\ref{cor:Psi-Phi-0} and Remark~\ref{rem:stilltrue} below.

It remains to check property (4).
We start with the following observation which is a direct consequence of the definitions:
for every pair of arrows $a$ and $b$, we have
\begin{equation}
\label{eq:Delta-b-partial-a-constant-term}
\Delta_a \partial_b S \equiv   e_{t(b)} \otimes \delta_{h(a), t(b)} \partial_{ba} S
\,\, {\rm mod} \,\, {\mathfrak m} \widehat{\otimes} R\langle\langle A\rangle\rangle  \ .
\end{equation}
In view of \eqref{eq:gamma-ab} and \eqref{eq:Psi-ba}, it follows
that, for every $\eta_a \in \Hom_\C (M(t(a)), M(h(a)))$, the morphism
$\Psi_{b,a}(\eta_a) - \delta_{h(a), t(b)} \eta_a \circ \gamma_{a,b}
\in \Hom_\C (M(h(b)), M(t(b)))$ is a linear combination of the
morphisms of the form $u_M \circ \eta_a \circ v_M$ with
$u \in \mathfrak m, \, v \in R\langle\langle A\rangle\rangle$.

Since ${\mathfrak m}$ acts nilpotently
on $M$, we have a descending filtration of $R$-modules
$$
M\supset {\mathfrak m}M\supset \cdots \supset {\mathfrak
m}^{\ell}M=0\ .
$$
For $p = 0, \dots, \ell - 1$, choose a $R$-submodule
$M^{(p)}$ in $M$ such that
$${\mathfrak m}^p M = M^{(p)} \oplus {\mathfrak m}^{p+1} M \ .$$
For $s\in \C^\star$, define $\lambda(s) \in \End_R (M)$ as the
$R$-module automorphism of $M$ acting on each $M^{(p)}$ as
multiplication by $s^p$.
This definition makes it clear that
\begin{equation}
\label{eq:limit}
\lim_{s\to 0} \lambda(s) \circ u_M \circ \lambda(s)^{-1} = 0
\end{equation}
for each $u \in {\mathfrak m}$.

Now for each $s\in \C^\star$, define the linear map
$$
\Psi^{(s)} =(\Psi_{b,a}^{(s)}):
U = \bigoplus_a \Hom_\C (M(t(a)),M(h(a)))\to \bigoplus_b\Hom_\C (M(h(b)),M(t(b)))
$$
by setting
$$
\Psi^{(s)}_{b,a}(\eta_a)=\lambda(s) \circ \Psi_{b,a}(\lambda(s)^{-1} \circ \eta_a).
$$
Since $\Psi^{(s)}$ is obtained from $\Psi$ by composing it
with invertible linear maps on both sides,
we have $\rank \Psi^{(s)}=\rank \Psi$ for all $s\in \C^\star$, and
more generally, $\dim \Psi^{(s)} (U') = \dim \Psi (U')$ for any
subspace $U' \subseteq U$ invariant under the automorphism
$(\eta_a) \mapsto (\lambda(s)^{-1} \circ \eta_a)$.
Note that the subspace $U_1 \subseteq U$ given by \eqref{eq:U1}
satisfies this condition; indeed, under the identification of $U$
with $\bigoplus_a \Hom_\C (M(t(a)),M(h(a)))$, $U_1$
identifies with $\bigoplus_a \Hom_\C (M(t(a)), {\mathfrak m} M(h(a)))$.

Now consider the linear map
$\Psi^{(0)} = \lim_{s\to 0}\Psi^{(s)}$.
Since under the continuous deformation
the rank of a linear map depends semi-continuously on the
deformation parameter, we conclude that
$$\dim \Psi (U_1) \geq \dim \Psi^{(0)} (U_1);$$
to finish the proof of property (4) in Lemma~\ref{lem:two-maps},
it suffices to show that
\begin{equation}
\label{eq:dim-Psi-0-U1}
\dim \Psi^{(0)} (U_1) =  \sum_i {\rm rk} \gamma_i \cdot {\rm rk} \alpha_i \ .
\end{equation}
In view of \eqref{eq:limit}, each component $\Psi^{(0)}_{b,a}$ of $\Psi^{(0)}$ acts by
$$\Psi_{b,a}^{(0)}(\eta_a) = \delta_{h(a), t(b)} \eta_a \circ \gamma_{a,b}.$$
Using natural identifications
$$\bigoplus_a \Hom_\C (M(t(a)),M(h(a))) =
\bigoplus_i \Hom_\C (M_{\rm in}(i),M(i))\ ,$$
$$\bigoplus_b \Hom_\C (M(h(b)),M(t(b))) =
\bigoplus_i \Hom_\C (M_{\rm out}(i),M(i)) \ ,$$
the operator $\Psi^{(0)}$ translates into the direct sum of
operators
$$\Psi^{(0)}_i: \Hom_\C (M_{\rm in}(i),M(i)) \to \Hom_\C (M_{\rm out}(i),M(i))$$
acting by
$$\Psi^{(0)}_i (\eta_i) = \eta_i \circ \gamma_i \ .$$
This description makes \eqref{eq:dim-Psi-0-U1} clear, finishing
the proofs of Lemma~\ref{lem:two-maps} and
Theorem~\ref{th:E-lower-bound-beta}.
\end{proof}

The following corollary is immediate from
\eqref{eq:E-lower-bound-beta-abbreviated}.

\begin{corollary}
\label{cor:dkplus-dkminus}
Suppose ${\mathcal M}=(M,V)$ is a QP-representation such that $E({\mathcal M})=0$.
Then for every vertex $k$ we have:
\begin{enumerate}
\item Either $M(k) = \{0\}$ or $V(k) = \{0\}$;
\item Either $\ker \beta_k = \{0\}$ or $\im \beta_k = \ker\gamma_k$.
\end{enumerate}
\end{corollary}

Since $E({\mathcal M})$ is invariant under mutations,
Corollary~\ref{cor:dkplus-dkminus} implies that if $E({\mathcal
M})=0$ then every QP-representation obtained from $\M$ by a
sequence of mutations satisfies the properties (1) and (2).
The following example shows that the converse is not true.

Let $Q$ be the Kronecker quiver
$$
\xymatrix{
1 \ar@<.5ex>[r]^{a, b}\ar@<-,5ex>[r] & 2}.
$$
For every positive integer~$n$, let $\M_n = (M_n, \{0\})$ be the indecomposable positive
QP-representation of $(Q,0)$ such that $M_n(1) = M_n(2) = \C^n$,
and the linear maps $a_{M_n}$ and $b_{M_n}$ from $M_n(1)$ to
$M_n(2)$ are as follows: $a_{M_n} = I$ is the identity map, while
$b_{M_n} = J$ is the nilpotent Jordan $n$-block.
Recalling \eqref{eq:g-M}, we see that the $\g$-vector of $\M_n$
is equal to $(n,-n)$.
Also $\Hom_Q(M_n, M_n)$ is naturally isomorphic to the centralizer
of $J$ in ${\rm End}(\C^n)$, hence we have $\langle M_n, M_n \rangle = n$.
Recalling \eqref{eq:E-abbreviated}, we get
$$E(\M_n) = \langle M_n, M_n \rangle + d_1(M_n) g_1(\M_n) +
d_2(M_n) g_2(\M_n) = n + n^2 - n^2 = n \ .$$
On the other hand, it is easy to see that $\M_n$ as well as all
representations obtained from it by mutations, satisfy properties
(1) and (2) in Corollary~\ref{cor:dkplus-dkminus} (in fact, every
representation obtained from $\M_n$ by mutations is either
right-equivalent to $\M_n$, or differs from it just by
interchanging vertices $1$ and $2$).

\section{Applications to cluster algebras}
\label{sec:proof-cluster-conjectures}

\begin{proof}[Proof of Theorem~\ref{th:fg-conjectures}]
We fix $t_0, t \in \TT_n$, a skew-symmetric integer $n \times n$
matrix~$B$, and a nondegenerate potential~$S$ on the quiver $Q = Q(B)$.
Recall that in Theorem~\ref{th:g-F-rep} the $\g$-vector
$\g_{\ell;t}^{B;t_0}$ and the $F$-polynomial $F_{\ell;t}^{B;t_0}$
from the theory of cluster algebras are interpreted as the
$\g$-vector $\g_\M$ and the $F$-polynomial $F_\M$ associated with
the QP-representation $\M = \M_{\ell;t}^{B;t_0}$ of $(Q,S)$ given
by \eqref{eq:M-cluster-variable}.

Conjectures~\ref{con:F-CT-1} and \ref{con:F-HT-1} are immediate
from this interpretation, see Proposition~\ref{pr:F-ct-ht}.

Our next target is Conjecture~\ref{con:g-transition}.
Comparing the desired formula
\eqref{eq:Langlands-dual-trop} with \eqref{eq:g-transition},
we see that it is enough to prove the equality
\begin{equation}
\label{eq:hk-gk}
\min(0, g_k) = h_k \ ,
\end{equation}
where $g_k$ and $h_k$ are given by \eqref{eq:gi-abbreviated}
and \eqref{eq:h-ker-beta}, respectively.
Substituting these expressions into \eqref{eq:hk-gk} and adding
$\dim \ker \beta_k$ to both sides, we arrive at the equality
\begin{equation}
\label{eq:hk-gk-rep}
\min(\dim (\ker \beta_k), \dim (\ker \gamma_k/ {\rm im} \ \beta_k) + d_k^-(\M)) = 0 \ .
\end{equation}
To finish the proof, it remains to observe that, in view of
Corollary~\ref{cor:EMzero}, the QP-representation $\M$ satisfies
properties (1) and (2) in Corollary~\ref{cor:dkplus-dkminus}, and
that \eqref{eq:hk-gk-rep} clearly holds for any representation
with these properties.

Now we are ready to prove Conjecture~\ref{con:signs-gi}.
The key observation is that the above argument proves the equality
\eqref{eq:hk-gk} not only for each representation
$\M_{\ell;t}^{B;t_0}$ but also for the direct sum
$$\M = \M_{1;t}^{B;t_0} \oplus \cdots \oplus \M_{n;t}^{B;t_0}$$
(since $\M$ satisfies the assumption in
Corollary~\ref{cor:EMzero}).
In view of Lemma~\ref{lem:g-F-mutation}, we have
$g_k = h_k - h'_k$, where $h'_k = - \dim \ker \overline \beta_k$ is the $k$-th component of the
vector $\h_{\overline \M}$ for the QP-representation $\overline \M = \mu_k(\M)$.
Thus, \eqref{eq:hk-gk} is equivalent to
\begin{equation}
\label{eq:hk-h-prime-k}
\max(h_k, h'_k) = 0 \ .
\end{equation}
Suppose that $h_{k} = 0$, that is, $\ker \beta_k = 0$.
Then the same property holds for each direct summand
$\M_\ell = \M_{\ell;t}^{B;t_0}$, implying that
$$g_k (\M_\ell) = - h_k (\overline {\M_\ell}) \geq 0$$
for all $\ell = 1, \dots, n$.
This shows that the $k$-th coordinates of the vectors
$\g_{1;t}^{B;t_0},\dots,\g_{n;t}^{B;t_0}$ are nonnegative.
If $h'_{k}=0$, then the same argument shows that
the $k$-th coordinates of all these vectors are nonpositive,
finishing the proof of Conjecture~\ref{con:signs-gi}.

As for Conjecture~\ref{con:g-vector-Z-basis}, it is an easy
consequence of already proven Conjectures~\ref{con:g-transition}
and \ref{con:signs-gi} combined with the following observation
made already in \cite[Remark~7.14]{ca4}:
\begin{equation}
\label{eq:sign-coherent-tropical-Langlands-imply-basis}
\vcenter{\hsize=5in\noindent
if $\g_1, \dots, \g_n$ are sign-coherent
vectors forming a $\Z$-basis in~$\Z^n$,
then the transformation
\eqref{eq:Langlands-dual-trop} sends them to  a $\Z$-basis in
$\Z^n$.}
\end{equation}
Indeed, to show that the vectors $\g_{1;t}^{B;t_0}, \dots, \g_{n;t}^{B;t_0}$
form a $\Z$-basis of~$\Z^n$, proceed by induction on the distance
between $t_0$ and $t$ in $\TT_n$.
The basic step $t=t_0$ is clear from (\ref{eq:gg-initial}), and
the inductive step follows from
\eqref{eq:sign-coherent-tropical-Langlands-imply-basis}.

To finish the proof of Theorem~\ref{th:fg-conjectures}, it remains
to prove Conjecture~\ref{con:g-vectors-separate-cluster-monomials}.

\begin{lemma}
\label{lem:negative-reps-characterization}
For a QP-representation $\M = (M,V)$, the following conditions are
equivalent:
\begin{enumerate}
\item $\M$ is negative, i.e., $M = \{0\}$.
\item $E(\M) = 0$, and the $\g$-vector $\g_\M = (g_1, \dots, g_n)$ is
nonnegative.
\end{enumerate}
Under these conditions, we have $\dim V(i) = g_i$ for all~$i$, so
$\M$ is uniquely determined by its $\g$-vector.
\end{lemma}

\begin{proof}
The only non-trivial statement is the implication $(2)
\Longrightarrow (1)$.
We have already established that the equality $E(\M) = 0$ implies
\eqref{eq:hk-gk}, so if $\g_\M$ is nonnegative then $h_k = 0$ for all~$k$.
Thus we have $\ker \beta_k = 0$ for all~$k$.
It remains to observe that the latter condition cannot hold for a
nonzero nilpotent quiver representation~$M$.
Indeed, if ${\mathfrak m}^{\ell-1}M \neq 0$, and ${\mathfrak m}^{\ell}M = 0$
for some $\ell \geq 1$, then
$0 \neq {\mathfrak m}^{\ell-1}M \subseteq \oplus_k \ker \beta_k$,
finishing the proof.
\end{proof}

\begin{lemma}
\label{lem:M-M-prime}
Let $\M$ and $\M'$ be QP-representations of the same nondegenerate
QP, and suppose that $\M'$ is mutation-equivalent to a negative
representation.
The following conditions are equivalent:
\begin{enumerate}
\item $\M$ is right-equivalent to $\M'$.
\item $E(\M) = 0$, and $\g_\M = \g_{\M'}$.
\end{enumerate}
\end{lemma}

\begin{proof}
Again only the implication $(2) \Longrightarrow (1)$ needs a proof.
Since $E(\M) = E(\M')= 0$, the already established formula \eqref{eq:Langlands-dual-trop}
shows that the $\g$-vectors remain the same under applying to $\M$
and $\M'$ the same sequence of mutations.
Since mutations also preserve right-equivalence, in proving that $(2) \Longrightarrow (1)$
we may assume that $\M'$ is negative, in which case the statement
follows from Lemma~\ref{lem:negative-reps-characterization}.
\end{proof}

Under the assumptions of
Conjecture~\ref{con:g-vectors-separate-cluster-monomials},
consider the QP-representations
$$
{\mathcal M}=\bigoplus_{i\in I} ({\mathcal M}_{i;t}^{B;t_0})^{a_i},
\quad {\mathcal M}'=\bigoplus_{i\in I'} ({\mathcal M}_{i;t'}^{B;t_0})^{a'_i} \ .
$$
In view of \eqref{eq:g-sum}, we have
$$
\g_{\mathcal M}=\sum_{i\in I} a_i\g^{B;t_0}_{i;t}=
\sum_{i\in I'}a_i \g^{B;t_0}_{i;t'} = \g_{{\mathcal M}'} \ .
$$
Also we have $E(\M) = E(\M')= 0$ since both $\M$ and $\M'$ are
mutation-equivalent to negative QP-representations.
By Lemma~\ref{lem:M-M-prime}, $\M$ is right-equivalent to $\M'$.
Because of the uniqueness of the decomposition into indecomposables,
there exists a bijection $\sigma:I\to I'$
such that ${\mathcal M}_{i;t}^{B;t_0}$ is right-equivalent to
${\mathcal M}_{\sigma(i);t'}^{B;t_0}$, and $a_i = a'_{\sigma(i)}$
for $i \in I$.
Thus we have
$$
\g^{B;t_0}_{i;t}=\g^{B;t_0}_{\sigma(i);t'}, \quad
F^{B;t_0}_{i;t}=F^{B;t_0}_{\sigma(i);t'}
$$
for all $i \in I$, finishing the proofs of
Conjecture~\ref{con:g-vectors-separate-cluster-monomials} and
of Theorem~\ref{th:fg-conjectures}.
\end{proof}

\section{Homological interpretation of the $E$-invariant}
\label{sec:homological}

Throughout this section we fix a quiver $Q$ without oriented $2$-cycles,
and a QP $(Q,S)$.
Let $\M =(M, V)$ and $\N = (N, W)$ be two QP-representations of $(Q,S)$.
Our aim is to associate to $\M$ and $\N$ a vector space $\E^\inj(\M,\N)$
such that $\dim \E^\inj(\M,\N) = E^\inj(\M,\N)$, the integer function
defined in \eqref{eq:E-abbreviated}.
It will be more convenient for us to work with the ``twisted"
function $E^\proj(\M,\N) = E^\inj(\N^\star, \M^\star)$, where
$\M^\star$ and $\N^\star$ are QP-representations of the opposite
QP $(Q^{\rm op}, S^{\rm op})$ constructed before
Proposition~\ref{pr:E-invariant-star}.
Clearly, we have $\langle N^\star, M^\star \rangle = \langle M, N \rangle$,
implying
\begin{equation}
\label{eq:E-star}
E^\proj(\M, \N) = \langle M, N \rangle + \sum_{k \in Q_0} g_k(\M^\star) d_k(N).
\end{equation}

We turn to the construction of a vector space $\E^\proj(\M,\N)$ such that
\begin{equation}
\label{eq:dim-E-star}
\dim \E^\proj(\M,\N) = E^\proj(\M,\N).
\end{equation}
Let $\Pcal(Q,S)$ be the Jacobian  algebra of $(Q,S)$ (see
Section~\ref{sec:F-poly}).
In the rest of the section we assume that
\begin{equation}
\label{eq:jacobian-alg-fin-dimensional}
\vcenter{\hsize=5in\noindent
the potential $S$ belongs to the path algebra $R \langle A \rangle$, and the
two-sided ideal
$J_0$ in $R \langle A \rangle$ generated by all cyclic derivatives
$\partial_a S$ contains some power $\mathfrak{m}^N$.}
\end{equation}
(Recall that in our general setup, $S$ belongs to the completed
path algebra $R \langle \langle A \rangle \rangle$, and the
Jacobian ideal $J$ of $S$ is the closure of $J_0 $ in
$R \langle \langle A \rangle \rangle$.)
Under this assumption, the Jacobian algebra
$\Pcal(Q,S) = R \langle \langle A \rangle \rangle/J$ is identified
with $R \langle A \rangle/J_0$, and it is finite-dimensional.
In this situation, all the $\Pcal(Q,S)$-modules considered below
will be finite-dimensional as well.

For every vertex $k \in Q_0$ let $P_k$ denote the indecomposable projective
$\Pcal(Q,S)$-module corresponding to~$k$.
Recall that $P_k$ is given by
\begin{equation}
\label{eq:projectives}
P_k = \bigoplus_{i \in Q_0} \Pcal(Q,S)_{i,k},
\end{equation}
where the double $Q_0$-grading on $\Pcal(Q,S)$ comes from the
$R$-bimodule structure (see Section~\ref{sec:QP-background}).
In particular, each $P_k$ is finite-dimensional in view of
\eqref{eq:jacobian-alg-fin-dimensional}.


To every (finite-dimensional) $\Pcal(Q,S)$-module
$M$, we associate the sequence of $\Pcal(Q,S)$-module homomorphisms
\begin{equation}
\label{eq:canpres}
\bigoplus_{b \in Q_1} (P_{t(b)}\otimes M(h(b))) \buildrel{\psi}\over\rightarrow
\bigoplus_{a\in Q_1} (P_{h(a)}\otimes M(t(a))) \buildrel{\varphi}\over\rightarrow
\bigoplus_{k\in Q_0} (P_k\otimes M(k)) \buildrel{\rm ev}\over\rightarrow M\rightarrow 0
\end{equation}
defined as follows.
The $\Pcal(Q,S)$-module homomorphisms ${\rm ev}$ and $\varphi$ are given by
\begin{equation}
\label{eq:evaluation-morphism}
{\rm ev}(p\otimes m)=pm \quad (p\in P_k,\, m\in M(k)),
\end{equation}
and
\begin{equation}
\label{eq:varphi-first-form}
\varphi(p\otimes m)= pa\otimes m - p\otimes a_M (m) \quad (p\in P_{h(a)},\, m\in M(t(a))),
\end{equation}
while the component
$\psi_{a,b}: P_{t(b)}\otimes M(h(b)) \to P_{h(a)}\otimes M(t(a))$
of $\psi$ is given by (in the notation of \eqref{eq:delta-partial-S} and \eqref{eq:Psi-ba})
\begin{equation}
\label{eq:psi-ab}
\psi_{a,b}(p \otimes m) = \sum_\nu pu_{b,a}^{(\nu)} \otimes (v_{a,b}^{(\nu)})_M \cdot m
\ .
\end{equation}

\begin{proposition}
\label{pr:extended-canonical-presentation}
The sequence \eqref{eq:canpres} is exact.
\end{proposition}

\begin{proof}
As pointed out by the referee,
this proposition follows from the results of \cite{bk}.
For the convenience of the reader we present some details (also kindly provided by the referee).
To make our notation closer to that of \cite{bk}, in the following argument
we denote the Jacobian algebra $R\langle A\rangle/J_0$ by $\Lambda$, and rename $J_0$ into $I$.

The ring $\Lambda$ is a bimodule over itself.
If we splice the exact sequence \cite[(1.4)]{bk} for $n=0$ and $n=1$ together, we get a
bimodule resolution of $\Lambda$ as follows
\begin{equation}\label{eq:Lambda}
\xymatrix{
\Lambda\otimes\displaystyle  \frac{I}{I{\mathfrak m}+{\mathfrak m}I}\otimes \Lambda\ar[rr]^{d_2} & &
\Lambda\otimes A \otimes \Lambda\ar[r]^{d_1}  & \Lambda\otimes\Lambda\ar[r]^{d_0} & \Lambda\ar[r] & 0.}
\end{equation}
Here the tensor products are over $R$, and we have identified $A$ with ${\mathfrak m}/{\mathfrak m}^2$. Note that $I/(I{\mathfrak m}+{\mathfrak m}I)$ is spanned by all partial derivatives
of the potential. The differential $d_2$ and $d_1$ are given after (1.3) in \cite{bk}.
Define
$$
\mu:R\langle A\rangle\to R\langle A\rangle\otimes A\otimes R\langle A\rangle
$$
by $\mu(a_1a_2\cdots a_s)=\sum_{i=1}^s a_1\cdots a_{i-1}\otimes a_i\otimes a_{i+1}\cdots a_s$
(this map is denoted by $\Delta$ in \cite{bk}).
Then $d_2$ sends the residue class of an element $1\otimes u\otimes 1$ to $\Delta(u)$.
The partial derivative $\partial_\xi$ was defined for $\xi\in A^\star$ in \cite[(3.1)]{dwz}.
By identifying $A$ with $A^\star$ using a basis of arrows, we have defined $\partial_b$
for an arrow $b\in A$.
We have a surjection
\begin{equation}\label{eq:surj}
A^\star \to \frac{I}{I{\mathfrak m}+{\mathfrak m}I}
\end{equation}
defined by $\xi\mapsto \partial_\xi S+I{\mathfrak m}+{\mathfrak m}I$.
We can replace the module on the left in (\ref{eq:Lambda}) by
$\Lambda\otimes A^\star\otimes \Lambda$ using (\ref{eq:surj}).
Therefore, we have an exact sequence:
\begin{equation}\label{eq:Lambda2}
\xymatrix{
\Lambda\otimes A^\star \otimes \Lambda\ar[r]&
\Lambda\otimes A \otimes \Lambda\ar[r]  & \Lambda\otimes\Lambda\ar[r] & \Lambda\ar[r] & 0.}
\end{equation}
If we apply the functor $\bullet \otimes_{\Lambda}M$ to (\ref{eq:Lambda2}), we obtain (\ref{eq:canpres}).
Note that the sequence remains exact after applying the functor, because (\ref{eq:Lambda2})
 splits as a sequence of right $\Lambda$-modules.



\end{proof}

\begin{corollary}
\label{cor:Psi-Phi-0}
The maps $\Phi$ and $\Psi$ given by \eqref{eq:Phi} and \eqref{eq:Psi-ba}
satisfy the condition $\Psi \circ \Phi = 0$.
\end{corollary}

\begin{proof}
Note that, for every $\Pcal(Q,S)$-module $M$, a vector
space $U$, and a vertex $k \in Q_0$, there is a natural isomorphism
\begin{equation}
\label{eq:hom-P-M}
\Hom_\C (U, M(k)) \to
\Hom_{\Pcal(Q,S)}(P_k \otimes U, M)
\end{equation}
sending $\sigma \in \Hom_\C (U, M(k))$ to the composed morphism
$$P_k \otimes U \buildrel{{\rm id} \otimes \sigma}\over\longrightarrow
P_k\otimes M(k) \buildrel{\rm ev}\over\rightarrow M \ .$$
An easy check shows that the maps $\Phi$ and $\Psi$ are obtained
from the maps $\varphi$ and $\psi$ given by \eqref{eq:varphi-first-form}
and \eqref{eq:psi-ab} by applying the contravariant functor
$\Hom_{\Pcal(Q,S)}(-, M)$ and using the isomorphism in
\eqref{eq:hom-P-M}.
So the statement in question follows from the exactness of (\ref{eq:canpres}). \end{proof}
\begin{remark}\label{rem:stilltrue}
Corollary~\ref{cor:Psi-Phi-0} is still true without the assumption~(\ref{eq:jacobian-alg-fin-dimensional}). In the more general case, the modules $P_i$ may be infinite dimensional, but the composition $\varphi\circ\psi$ is still equal to 0 in (\ref{eq:canpres}).
\end{remark}

The sequence \eqref{eq:canpres} produces a \emph{presentation} of
the $\Pcal(Q,S)$-module~$M$, which can be rewritten as
\begin{equation}
\label{eq:2canpres}
\bigoplus_{k\in Q_0} (P_{k}\otimes M_{\rm in}(k)) \buildrel{\varphi}\over\rightarrow
\bigoplus_{k\in Q_0} (P_k\otimes M(k)) \buildrel{\rm ev}\over\rightarrow M\rightarrow 0,
\end{equation}
where with some abuse of notation we use the same symbol $\varphi$
for the leftmost map: this map is now given by
\begin{equation}
\label{eq:canonical-map}
\varphi(p\otimes m)= \sum_{h(a) = k} (pa\otimes {\rm pr}_{t(a)} m) - p\otimes \alpha_{k}(m)\quad
(p\in P_k, \,m\in M_{\rm in}(k))
\end{equation}
(here ${\rm pr}_{t(a)}$ stands for the projection
$M_{\rm in}(k) = \bigoplus_{h(a) = k} M(t(a)) \to M(t(a))$).

We claim that the presentation \eqref{eq:2canpres} can be truncated as follows.
For every $k \in Q_0$, choose subspaces $U'_k, U''_k \subseteq M_{\rm in}(k)$
and $M^{(0)}(k) \subseteq M(k)$ such that
\begin{equation}
\label{eq:complements}
\ker(\alpha_k) = {\rm im}(\gamma_k)  \oplus U'_k, \,\,
M_{\rm in}(k) = \ker(\alpha_k) \oplus U''_k, \,\,
M(k) = M^{(0)}(k) \oplus {\rm im}(\alpha_k),
\end{equation}
and consider the projective $\Pcal(Q,S)$-modules
\begin{align}
\label{eq:P0-P1}
P' &= \bigoplus_{k\in Q_0} (P_k\otimes \ker(\alpha_{k})),
\quad P'' = \bigoplus_{k\in Q_0} (P_k\otimes U''_k),\\
\nonumber
P^{(1)} &= \bigoplus_{k\in Q_0} (P_k\otimes U'_k),
\quad P^{(0)} = \bigoplus_{k\in Q_0} (P_k\otimes M^{(0)}(k)).
\end{align}

\begin{proposition}\

\label{pr:min-presentation}
\begin{enumerate}
\item For every $p' \in P'$, there exists a unique $p'' \in P''$ such
that $\varphi(p' - p'') \in P^{(0)}$.
The map $\overline \varphi: P' \to P^{(0)}$ given by $\overline \varphi(p') = \varphi(p' - p'')$
is a $\Pcal(Q,S)$-module homomorphism.
\item The restrictions of $\overline \varphi$ to $P^{(1)}$ and of ${\rm ev}$ to $P^{(0)}$
make the sequence
\begin{equation}
\label{eq:min-presentation}
P^{(1)} \buildrel{\overline \varphi}\over\to P^{(0)}  \buildrel{\rm ev}\over\to M \to 0
\end{equation}
exact, thus giving a presentation of~$M$.
\item The presentation \eqref{eq:min-presentation} is minimal,
that is, the map $\overline \varphi: P^{(1)} \to P^{(0)}$ induces
an isomorphism
$P^{(1)}/{\mathfrak m} P^{(1)} \to
{\rm im}(\overline \varphi)/{\mathfrak m}\, {\rm im}(\overline \varphi)$,
where ${\mathfrak m}$ is the maximal ideal in $\Pcal(Q,S)$.
\end{enumerate}
\end{proposition}

Before proving Proposition~\ref{pr:min-presentation}, we
use it to construct the space $\E^\proj(\M,\N)$
(for any QP-representations $\M =(M, V)$ and $\N = (N, W)$ of $(Q,S)$)
satisfying \eqref{eq:dim-E-star}.
Note that the $\Pcal(Q,S)$-module homomorphism $\overline \varphi: P^{(1)} \to P^{(0)}$ in
Proposition~\ref{pr:min-presentation} induces a $\C$-linear map
$$\overline \varphi^\star: \Hom_{\Pcal(Q,S)} (P^{(0)}, N)\rightarrow
\Hom_{\Pcal(Q,S)} (P^{(1)}, N).$$
We now define the space $\E^\proj(M,N)$ as the cokernel of
$\overline \varphi^\star$, that is, from an exact sequence
\begin{equation}
\label{eq:exact-seq-E-star}
\Hom_{\Pcal(Q,S)} (P^{(0)}, N) \buildrel{\overline \varphi^\star}\over\rightarrow
\Hom_{\Pcal(Q,S)} (P^{(1)}, N) \rightarrow \E^\proj(M,N)
\rightarrow 0 \ .
\end{equation}
And finally we set
\begin{equation}
\label{eq:full-E-star}
\E^\proj(\M,\N) = \E^\proj(M,N) \oplus \Hom_R(V,N)\ .
\end{equation}

\begin{theorem}
\label{th:E-star-as-dimension}
The space $\E^\proj(\M,\N)$ satisfies
\eqref{eq:dim-E-star}, i.e., its dimension is given by \eqref{eq:E-star}.
\end{theorem}

\begin{proof}
Using the presentation \eqref{eq:min-presentation}, we include
\eqref{eq:exact-seq-E-star} into a longer exact sequence
\begin{align}
\label{eq:longer-exact-sequence-for-E-star}
0 &\to \Hom_{\Pcal(Q,S)} (M, N) \to
\Hom_{\Pcal(Q,S)} (P^{(0)}(M), N)\\
\nonumber
 &
 \rightarrow
\Hom_{\Pcal(Q,S)} (P^{(1)}(M), N) \rightarrow \E^\proj(M,N)
\rightarrow 0 \ .
\end{align}

Computing the dimensions of the terms in
\eqref{eq:longer-exact-sequence-for-E-star}, we get
\begin{align*}
&\dim \E^\proj(M,N) =\\
& \langle M, N \rangle - \dim \Hom_{\Pcal(Q,S)} (P^{(0)}(M), N)
+ \dim \Hom_{\Pcal(Q,S)} (P^{(1)}(M), N) =\\
&\langle M, N \rangle - \sum_{k \in Q_0} \dim {\rm
coker}(\alpha_{k;M})\cdot d_k(N) +
\sum_{k \in Q_0} \dim \frac{\ker(\alpha_{k;M})}{{\rm im}(\gamma_{k;M})} \cdot d_k(N) = \\
& \langle M, N \rangle + \sum_{k \in Q_0} (\dim M_{\rm in}(k) -
d_k(M) + {\rm rk}(\gamma_{k;M})) \cdot d_k(N)= \\
& \langle M, N \rangle + \sum_{k \in Q_0} (g_k(\M^\star) -
d_k^-(\M)) \cdot d_k(N)
\end{align*}
(for the last equality see \eqref{eq:gi-thru-rank-gamma});
note that in view of \eqref{eq:hom-P-M}, $\Hom_{\Pcal(Q,S)}(P_k, N)$ is naturally
isomorphic to $N(k)$, hence $\dim \Hom_{\Pcal(Q,S)}(P_k, N) = d_k(N)$.

To finish the proof of \eqref{eq:dim-E-star}, it remains to note
that
$$\Hom_R(V,N) = \bigoplus_{k \in Q_0} \Hom_\C(V(k),N(k)),$$
implying
$$\dim \Hom_R(V,N) = \sum_{k \in Q_0} d_k^-(\M) d_k(N).$$
Thus,
\begin{align*}
&\dim \E^\proj(\M,\N) = \dim \E^\proj(M,N) + \sum_{k \in Q_0} d_k^-(\M) d_k(N)\\
&= \langle M, N \rangle + \sum_{k \in Q_0} g_k(\M^\star) d_k(N)
= E^\proj(\M,\N)
\end{align*}
by \eqref{eq:E-star}.
\end{proof}

\begin{proof}[Proof of Proposition~\ref{pr:min-presentation}]
We start by showing that the map ${\rm ev}: P^{(0)} \to M$ is surjective.
This is a special case of the following lemma.

\begin{lemma}
\label{lem:surjectivity}
Suppose $\eta: K \to L$ is a surjection of finite-dimensional $\Pcal(Q,S)$-modules.
Suppose that $K = K' \oplus K''$ is the direct sum of two submodules,  and
that $\eta(K'')  \subseteq {\mathfrak m} L$. Then $\eta(K') = L$.
\end{lemma}

\begin{proof}
Choose the direct complement $L^{(0)}$ to ${\mathfrak m} L$ in $L$.
Then  $L^{(0)}$ generates $L$ as a $\Pcal(Q,S)$-module.
Indeed, we have
$$L = {\mathfrak m} L + L^{(0)} = {\mathfrak m}^2 L + {\mathfrak m}  L^{(0)} +  L^{(0)}
= \cdots = {\mathfrak m}^{N+1} L + {\mathfrak m}^N  L^{(0)}  + \cdots + L^{(0)}$$
for each $N \geq 0$; choosing $N$ big enough so that ${\mathfrak m}^{N+1} L = \{0\}$,
we see that $L = \Pcal(Q,S)  L^{(0)}$.
This argument also shows that ${\mathfrak m} L = {\mathfrak m} L^{(0)}$.

Since $\eta$ is a homomorphism of $\Pcal(Q,S)$-modules, to prove
that $\eta(K') = L$, it suffices to show that
$L^{(0)} \subseteq \eta(K')$.
Using the surjectivity of $\eta: K \to L$ and the inclusion
$\eta(K'')  \subseteq {\mathfrak m} L^{(0)}$, we get
\begin{align*}
L^{(0)} &\subseteq \eta(K') + {\mathfrak m} L^{(0)} \subseteq
\eta(K') + {\mathfrak m} \eta(K') + {\mathfrak m}^2  L^{(0)} \subseteq \cdots\\
&\subseteq \eta(K') + {\mathfrak m} \eta(K')
+  \cdots + {\mathfrak m}^N \eta(K') +  {\mathfrak m}^{N+1} L^{(0)}
= \eta(K') + {\mathfrak m}^{N+1} L^{(0)}
\end{align*}
for each $N \geq 0$, implying as above that $L^{(0)} \subseteq \eta(K')$.
\end{proof}

Now the fact that ${\rm ev}(P^{(0)}) = M$ follows by
applying Lemma~\ref{lem:surjectivity} to the map
${\rm ev}: \bigoplus_{k\in Q_0} (P_k\otimes M(k)) \to M$
in place of $\eta: K \to L$, and to the submodules $K'$ and $K''$
given by
\begin{equation}
\label{eq:K-spaces}
K' = P^{(0)}, \quad K'' =
\bigoplus_{k\in Q_0} (P_k\otimes {\rm im}(\alpha_k)).
\end{equation}

Continuing the proof of Proposition~\ref{pr:min-presentation}, we
adopt the notation in \eqref{eq:P0-P1} and \eqref{eq:K-spaces},
thus viewing $\varphi$ as a homomorphism of $\Pcal(Q,S)$-modules
$P' \oplus P'' \to K' \oplus K''$.
We write $\varphi$ as $\varphi^{(1)} + \varphi^{(0)}$ in
accordance with the decomposition in \eqref{eq:canonical-map}.
The following properties are immediate from
\eqref{eq:canonical-map}:
\begin{align}
\label{eq:phi0}
&  \vcenter{\hsize=5in \noindent
$\varphi^{(0)}|_{P'} = 0$, while the restriction of
$\varphi^{(0)}$ to $P''$ is an isomorphism between $P''$ and
$K''$;} \\
\label{eq:phi1}
&{\rm im}(\varphi^{(1)}) \subseteq {\mathfrak m}(K' \oplus K'').
\end{align}
We claim that these properties imply the following:
\begin{align}
\label{eq:K'+phi-P''}
&K'' \subseteq {\mathfrak m} K' + \varphi(P'');\\
\label{eq:K'-cap-phi-P''}
&K' \cap \varphi(P'') = \{0\};\\
\label{eq:phi-P''-injective}
&\text{the restriction of $\varphi$ to $P''$ is injective.}
\end{align}
Note that these facts imply Part (1) of
Proposition~\ref{pr:min-presentation}.
Indeed, by \eqref{eq:K'+phi-P''} and \eqref{eq:K'-cap-phi-P''}, we have
\begin{equation}
\label{eq:direct-sum-for-psi}
K' \oplus K'' = K' \oplus \varphi(P'').
\end{equation}
This allows us to define the map $\overline \varphi: P' \to K' = P^{(0)}$ as
the composition ${\rm pr}_1 \circ \varphi$, where ${\rm pr}_1$ is
the projection of $K' \oplus K''$ onto $K'$ along $\varphi(P'')$.
Using \eqref{eq:phi-P''-injective}, we see that $\overline \varphi$ is exactly
the map in Part (1) of Proposition~\ref{pr:min-presentation} (the
fact that $\overline \varphi$ is a $\Pcal(Q,S)$-module homomorphism is obvious since so are
$\varphi$ and ${\rm pr}_1$).

To prove \eqref{eq:K'+phi-P''}, we use \eqref{eq:phi0} and \eqref{eq:phi1} to get
\begin{align*}
K'' &= \varphi^{(0)}(P'') \subseteq \varphi^{(1)}(P'') + \varphi(P'')
\subseteq {\mathfrak m}K'  + \varphi(P'') + {\mathfrak m}K''\\
&\subseteq {\mathfrak m}K'  + \varphi(P'') +
{\mathfrak m}({\mathfrak m}K'  + \varphi(P'') + {\mathfrak m}K'')
\subseteq {\mathfrak m}K'  + \varphi(P'') + {\mathfrak m}^2 K''.
\end{align*}
Iterating, we see that
$K'' \subseteq {\mathfrak m}K'  + \varphi(P'') + {\mathfrak m}^N K''$
for all $N \geq 1$, implying the desired inclusion \eqref{eq:K'+phi-P''}.

To prove \eqref{eq:K'-cap-phi-P''} and \eqref{eq:phi-P''-injective}, suppose that $\varphi(p'') =
k'$ for some $k' \in K'$ and $p'' \in P''$.
Using \eqref{eq:phi0} and \eqref{eq:phi1}, we see that
$$k' - \varphi^{(0)}(p'') = \varphi^{(1)}(p'') \in {\mathfrak m}(K' \oplus K''),$$
implying that $k' \in {\mathfrak m}K'$ and $\varphi^{(0)}(p'') \in
{\mathfrak m}K''$.
Once again applying \eqref{eq:phi0}, we conclude that $p'' \in {\mathfrak m}P''$.
Iterating this argument, we conclude that  $k' \in {\mathfrak m}^N K'$ and
$p'' \in {\mathfrak m}^N P''$ for all $N \geq 1$,
hence $k' = p'' = 0$, implying both \eqref{eq:K'-cap-phi-P''}
and \eqref{eq:phi-P''-injective}.

Turning to the proof of Part 2 of
Proposition~\ref{pr:min-presentation}, we first show that the sequence
\begin{equation}
\label{eq:presentation-intermediate}
P' \buildrel{\overline \varphi}\over\to K'  \buildrel{\rm ev}\over\to M \to 0
\end{equation}
is exact.
The surjectivity of ${\rm ev}: K' \to M$ is already proved above,
so (using the exactness of \eqref{eq:2canpres}) it remains to show that
$\overline \varphi(P') = \varphi(P' \oplus P'') \cap K'$; but this is immediate
from the definition of~$\overline \varphi$.

To prove Part 2, it remains to show that the restriction of the map $\overline \varphi: P' \to K'$ to the submodule
$P^{(1)} \subseteq P'$ has the same image as~$\overline \varphi$.
Note that $P' = P^{(1)} \oplus P'_1$, where
$$P'_1 = \bigoplus_{k\in Q_0} (P_k\otimes {\rm im}(\gamma_k)).$$
By Lemma~\ref{lem:surjectivity}, it is enough to show that
\begin{equation}
\label{eq:psi-of-P'1-inclusion}
\overline \varphi(P'_1) \subseteq {\mathfrak m} \overline \varphi(P').
\end{equation}
It follows easily from \eqref{eq:direct-sum-for-psi} that
$${\mathfrak m} \overline \varphi(P') = {\mathfrak m}(\varphi(P' \oplus P'') \cap K')
=  {\mathfrak m} \varphi(P' \oplus P'') \cap K',$$
and also that
$$\varphi(P') = \varphi^{(1)}(P') \subseteq {\mathfrak m}(K' \oplus K'')
= {\mathfrak m} K' \oplus  {\mathfrak m} \varphi(P''),$$
implying the inclusion $\overline \varphi(P'_1) \subseteq \varphi(P'_1 \oplus  {\mathfrak m} P'')$.
We see that the desired inclusion
\eqref{eq:psi-of-P'1-inclusion} is a consequence of the following:
\begin{equation}
\label{eq:phi-of-P'1-inclusion}
\varphi(P'_1) \subseteq {\mathfrak m} \varphi(P' \oplus P'').
\end{equation}

To prove \eqref{eq:phi-of-P'1-inclusion}, it suffices to show that
$\varphi(e_k \otimes \gamma_k(m)) \in {\mathfrak m} \varphi(P' \oplus P'')$
for every $m \in M(h(b))$, where $b$ is an arrow with $t(b) = k$.
But this follows from the exactness of the sequence \eqref{eq:canpres}
(more precisely, from the fact that ${\rm im}(\psi) \subseteq
\ker(\varphi)$), since in view of \eqref{eq:Delta-b-partial-a-constant-term}
we have
\begin{equation}
\label{eq:psi-gamma-mod-m}
e_k \otimes \gamma_k(m) \equiv   \psi(e_k \otimes m)
\,\, {\rm mod} \,\, {\mathfrak m} (P' \oplus P'')  \ .
\end{equation}
This concludes the proof of Part 2 of
Proposition~\ref{pr:min-presentation}.

To prove Part 3, note that \eqref{eq:psi-gamma-mod-m} implies the inclusion
$$\ker(\varphi) = {\rm im}(\psi) \subseteq P'_1 +  {\mathfrak m} (P' \oplus P'').$$
Now suppose that $p \in P^{(1)}$ is such that
$\overline \varphi(p) \in \overline \varphi ({\mathfrak m} P^{(1)})$.
Remembering the definition of $\overline \varphi$, we conclude that
$$p \in {\mathfrak m} P^{(1)} + P'' + \ker(\varphi) \subseteq
{\mathfrak m} P^{(1)} \oplus P'_1 \oplus P''.$$
Therefore, $p \in {\mathfrak m} P^{(1)}$, finishing the proof of
Proposition~\ref{pr:min-presentation}.
\end{proof}

\begin{remark}
\label{rem:minimality}
The presentation \eqref{eq:min-presentation} is minimal by part (3) of Proposition~\ref{pr:min-presentation}.
Minimal presentations are unique up to isomorphism.
One can show that, up to an isomorphism, the presentation \eqref{eq:min-presentation}
does not depend on the choice of splitting subspaces in \eqref{eq:complements}.
\end{remark}

\begin{remark}
\label{rem:injectives}
To emphasize the dependence of indecomposable projective modules $P_k$ (for $k\in Q_0$)
on the underlying QP $(Q,S)$, we will denote them $P_k = P_k (Q, S)$.
The indecomposable injective $\Pcal(Q,S)$-modules $I_k = I_k (Q, S)$ can be defined by
going to the opposite QP:
\begin{equation}
\label{eq:injectives}
I_k (Q, S)= (P_k (Q^{\rm op}, S^{\rm op}))^\star.
\end{equation}
By this definition, there is a duality between projective and injective
$\Pcal(Q,S)$-modules: every exact sequence involving the modules $P_k$ gives
rise to the exact sequence (with the arrows reversed) involving the $I_k$.
In particular, the presentation \eqref{eq:min-presentation} gives
rise to a ``co-presentation"
$$0 \to M \to \bigoplus_{k\in Q_0} (I_k \otimes
\ker(\beta_{k})^\star) \to \bigoplus_{k\in Q_0} (I_k \otimes U_k^\star),$$
where $U_k$ is a direct complement of ${\rm im}(\beta_k)$ in $\ker(\gamma_k)$.
\end{remark}
Recall that $\E^\proj(M,N)$ is defined in
(\ref{eq:exact-seq-E-star}). We also define
$\E^\inj(M,N)=\E^\proj(N^\star,M^\star)$.

\begin{corollary}
\label{cor:ext-and-tau}
We have the following isomorphisms
\begin{eqnarray}
\E^\proj(M,N) &=& \Hom_{\Pcal(Q,S)}(N,\tau(M))^\star\\
\E^\inj(M,N) &=& \Hom_{\Pcal(Q,S)}(\tau^{-1}(N),M)^\star
\end{eqnarray}
where $\tau$ is the Auslander-Reiten translation functor (see e.g., \cite[Section~IV.2]{ass}).
\end{corollary}
\begin{proof}
For this proof we will rely on the book \cite{ass}. We should point out
that the authors of \cite{ass} use the convention that all modules
are right-modules unless stated otherwise, where we assume modules to be left modules
by default.
Let $\nu$ be the Nakayama functor (see \cite[Section III, Definition 2.8]{ass}) from $\Pcal(Q,S)$-modules to $\Pcal(Q,S)$-modules
defined by
$$
\nu(M)=\Hom_{\Pcal(Q,S)}(M,\Pcal(Q,S))^\star.
$$
This functor has the property that
$$
\nu(P_k)=I_k
$$
for every vertex $k$. In particular, we have an isomorphism
\begin{equation}\label{eq:isoPI}
\Hom_{\Pcal(Q,S)}(P,M)=\Hom_{\Pcal(Q,S)}(M,\nu(P))^\star
\end{equation}
for every projective module $P$ (see \cite[Lemma 2.1]{ass}).
Consider the minimal presentation \eqref{eq:min-presentation}.
It follows from \cite[Section IV, Proposition 2.4]{ass} that the sequence
\begin{equation}\label{eq:nuP}
0\to \tau(M)\to \nu(P^{(1)})\to \nu(P^{(0)})
\end{equation}
is exact.
If we apply $\Hom_{\Pcal(Q,S)}(N,\cdot)^\star$ to (\ref{eq:nuP}), then it follows from (\ref{eq:isoPI}) that
\begin{equation}
\Hom_{\Pcal(Q,S)}(P^{(0)},N) \to \Hom_{\Pcal(Q,S)}(P^{(1)},N)\to \Hom_{\Pcal(Q,S)}(N,\tau(M))^\star\to 0
\end{equation}
is exact.
It follows from (\ref{eq:longer-exact-sequence-for-E-star})
that $\E^\proj(M,N)=\Hom_{\Pcal(Q,S)}(N,\tau(M))^\star$.

We have $\tau(N^\star)=\tau^{-1}(N)^\star$.
So it follows that
\begin{multline*}
\E^\inj(M,N)=\E^\proj(N^\star,M^\star)=\Hom_{\Pcal(Q,S)^{\rm op}}(M^\star,\tau(N^\star))^\star=\\
=\Hom_{\Pcal(Q,S)^{\rm op}}(M^\star,\tau^{-1}(N)^\star)^\star=
\Hom_{\Pcal(Q,S)}(\tau^{-1}(N),M)^\star.
\end{multline*}
\end{proof}

\begin{remark}
Auslander-Reiten duality states that $\Ext^1_{\Pcal(Q,S)}(M,N)$ is isomorphic to
$\overline{\Hom}_{\Pcal(Q,S)}(M,N)^\star$, where $\overline{\Hom}_{\Pcal(Q,S)}(M,N)$
is equal to $\Hom_{\Pcal(Q,S)}(M,N)$  modulo the morphisms that factor through
injective modules (see \cite[IV.2, Theorem 2.13]{ass}).
We may view $\Ext^1_{\Pcal(Q,S)}(M,N)$ as a subspace of $\E^\proj(M,N)$.
If $M$ has projective dimension $\leq 1$ then we have equality by~\cite[IV.2, Corollary 2.14]{ass}.
Similarly, $\Ext^1_{\Pcal(Q,S)}(M,N)$ can be viewed as a subspace of $\E^\inj(M,N)$,
with equality when $N$ has injective dimension $\leq 1$.
\end{remark}

\newpage
\noindent\footnotesize Harm Derksen\\
{\sc Department of Mathematics, University of Michigan,
Ann Arbor, MI 48109, USA}\\
{\it E-mail address:} { \tt hderksen@umich.edu}\\[10pt]
Jerzy Weyman\\
{\sc Department of Mathematics, Northeastern University,
Boston, MA 02115, USA}\\
{\it E-mail address:} {\tt j.weyman@neu.edu}\\[10pt]
Andrei Zelevinsky\\
{\sc Department of Mathematics, Northeastern University,
Boston, MA 02115, USA}\\
{\it E-mail address:} {\tt andrei@neu.edu}

\end{document}